\documentclass[onefignum,onetabnum]{siamart190516}
\usepackage{amsmath,amssymb,amsxtra}
\usepackage{mathrsfs}
\usepackage{float,verbatim}
\usepackage{threeparttable,booktabs}
\usepackage{graphicx}
\usepackage{epstopdf}
\usepackage{mathtools}
\usepackage{subfig}
\usepackage{color}
\usepackage{algpseudocode}
\usepackage{multirow}
\usepackage{array}
\newtheorem{remark}{Remark}[section]

\newtheorem{example}{Example}[section]

\usepackage{graphicx}
\usepackage{dsfont}
\usepackage{epstopdf}
\graphicspath{{./pics/}}
\usepackage{booktabs}
\usepackage{threeparttable}
\usepackage{verbatim,bm}
\newcommand{\Om}{\Omega}
\def\p{\partial} 
\newcommand{\op}{\operatorname}
\newcommand{\dx}{\mathrm{d}x}
\newcommand{\ds}{\mathrm{d}s}
\def\T{\mathcal{T}}
\def\bold{\boldsymbol}

\numberwithin{equation}{section}

\allowdisplaybreaks

\usepackage{cancel}

\allowdisplaybreaks

\title{Adaptive Crouzeix-Raviart finite elements for the first eigenpair of $p$-Laplacian}

\author{Guanglian Li\thanks{Department of Mathematics, The University of Hong Kong, Hong Kong Special Administrative Region, China. The research of  this author was partially supported by Hong Kong RGC through
General Research Fund (project 17317122) and Early Career Scheme (project 27301921). ({\tt lotusli@maths.hku.hk})} \and Yueqi Wang\thanks{Department of Mathematics, The University of Hong Kong, Hong Kong Special Administrative Region, China. ({\tt u3007895@connect.hku.hk})} \and  Yifeng Xu\thanks{Corresponding author. Department of Mathematics and Scientific Computing Key Laboratory of Shanghai Universities, Shanghai Normal University, Shanghai 200234, China. The research of this
author was partially supported by the National Natural Science Foundation of China under grants 12250013, 12261160361 and 12271367. ({\tt yfxu@shnu.edu.cn, yfxuma@aliyun.com})}}

\date{\today}

\begin{document}

\maketitle

\begin{abstract}
In this paper, we propose and analyze an adaptive Crouzeix-Raviart finite element method for computing the first Dirichlet eigenpair of the $p$-Laplacian problem. We prove that the sequence of error estimators produced by the adaptive algorithm has a vanishing limit and that, starting from a fine initial mesh, the relevant sequence of approximate eigenvalues converges to the first eigenvalue and the distance in a mesh-dependent broken norm between discrete eigenfunctions and the set composed of relevant continuous eigenfunctions also tends to zero. The analysis hinges on establishing a compactness property for Crouzeix-Raviart finite elements over a sequence of adaptively generated meshes, which represents key theoretical challenges and novelties. We present numerical results to illustrate the advantage of the proposed algorithm. 
\end{abstract}


\begin{AMS}
{65N12, 65N25, 65N30, 65N50, 35P30}
\end{AMS}

\section{Introduction}
Let $\Omega$ be a bounded polyhedral connected domain in $\mathbb{R}^d$ ($d=2,3$), and fix $p\in (1,\infty)$.
Consider the following Dirichlet eigenvalue problem of the $p$-Laplacian operator:
\begin{equation}\label{diff_eq}
    \left\{\begin{aligned}
        -\mathrm{div}(|\boldsymbol{\nabla}u|^{p-2}\boldsymbol{\nabla}u)&=\lambda|u|^{p-2}u&&\text{in}~\Om,\\
        u&=0&&\text{on}~\p\Om.
    \end{aligned}
    \right.
\end{equation}
In the engineering literature, the $p$-Laplacian operator is often used to describe non-Newtonian fluids, turbulent flow, flow through porous media and power-law materials etc. 

The weak formulation of \eqref{diff_eq} is to find $(\lambda,u)\in
\mathbb{R}\times V:=W_0^{1,p}(\Om)$ such that
\begin{equation}\label{vp_eigen}
    \int_\Om|\boldsymbol{\nabla}u|^{p-2}\boldsymbol{\nabla}u\cdot \boldsymbol{\nabla}v\dx=\lambda\int_\Om|u|^{p-2}uv\dx,\quad\forall v\in V.
\end{equation}
The existing theory \cite{MR912211,MR2196811,MR1007505} states that there is a nondecreasing sequence of positive eigenvalues
$\{\lambda_n\}_{n\geq1}$ tending to $\infty$. Moreover, the first eigenvalue $\lambda_1$
is simple and isolated \cite{MR1007505}, and can be characterized as the minimum of the Rayleigh quotient:
\begin{equation}\label{min_cont}
    \lambda_1=\inf_{v\in V\setminus \{0\}}\mathcal{J}(v):=\int_\Om|\boldsymbol{\nabla}v|^p\dx/\int_\Om|v|^p\dx.
\end{equation}
By the direct method in calculus of variation \cite{MR2722059,MR1007505}, there exists a minimizer to problem \eqref{min_cont} among all $L^p$-normalized functions. This implies an $L^p$-normalized
eigenfunction set with respect to $\lambda_1$:
\[E_{\lambda_1}:=\left\{u\in
V~\Big|~\int_\Om|\boldsymbol{\nabla}u|^{p-2}\boldsymbol{\nabla}u\cdot
\boldsymbol{\nabla}v\dx=\lambda_1\int_\Om|u|^{p-2}uv\dx,~\forall v\in
V,~\|u\|_{L^p(\Omega)}=1\right\}.\]

The numerical treatment of problem \eqref{diff_eq} is  challenging, due to the degeneracy of the differential operator and the possible presence of geometric singularities of the domain $\Omega$, e.g., reentrant corners. One effective strategy to resolve the challenges is to employ a suitable adaptive technique, e.g., adaptive finite element method (AFEM). The AFEM can achieve the desired accuracy with minimal degrees of freedom, making it a popular tool in scientific and engineering computing. The standard AFEM typically consists of the following four modules in every loop:
\begin{equation}\label{afemloop}
\textrm{SOLVE}\to\textrm{ESTIMATE}\to\textrm{MARK}\to\textrm{REFINE}.
\end{equation}
Since its first introduction in \cite{MR483395}, there has been significant progress in mathematical analysis of AFEM. The use of a posteriori estimation, a key component in the \textrm{ESTIMATE} module, is well-documented \cite{MR1885308,MR3059294}. Over the past three decades, there has been substantial progress on the convergence and complexity of AFEM; see the reviews \cite{MR4793681,MR3170325,MR2648380} and the references therein. 

In this work, we develop a novel adaptive nonconforming linear FEM, using the Crouzeix-Raviart (C-R) FEM \cite{MR343661} for computing the first eigenpair of \eqref{vp_eigen} (cf. Algorithm \ref{anfem}), and establish the convergence of the sequence of discrete eigenpairs $\{(\mu_k,u_k)\}_{k\geq 0}$ (cf. Theorems \ref{thm:conv2} and \ref{thm:conv_est}). We derive a computable quantity inspired by \cite{MR2783229} that measures the discontinuity of the discrete eigenfunction $u_k$, and prove the weak convergence of the adaptively generated sequence $\{u_k\}_{k \geq 0}$ up to a subsequence (see Lemma \ref{lem:weakconv}) if the computable quantity tends to zero as $k\to\infty$. Moreover, the $L^p$ weak convergence in Lemma \ref{lem:weakconv} can be lifted to a strong one (see Lemma \ref{lem:L^p-strongconv}). This is the discrete compactness of C-R elements over a sequence of adaptively generated meshes, thereby generalizing the result in \cite{MR566091} (see also \cite[Chapter 10.3]{ShiWang:2013}) in the case of uniform mesh refinements, and the compactness property of adaptive C-R elements in the $L^2$-version \cite{JinLiXuZhu:2025}.

Moreover, we prove that a residual functional associated with the discrete eigenpair $(\mu_k,u_k)$ vanishes in the limit as a computable quantity related to \eqref{vp_disc} tends to zero (see Lemma \ref{lem:residue}). This result relies on an interpolation operator from $W_{0}^{1,p}(\Omega)$ to the C-R FE space \cite{MR343661,MR2783229}, which also provides a guaranteed lower bound for $\lambda_1$ in Theorem \ref{thm:GLB}. Then we establish the existence of a strong limit pair $ (\mu , w) \in \mathbb{R} \times W_0^{1,p}(\Omega)$ of a subsequence $\{(\lambda_{k_j},u_{k_j})\}_{j\geq 0}$, provided that two relevant computable quantities converge to zero (see Theorem \ref{thm:strong-conv}). With an appropriate assumption on the initial mesh, this gives the desired convergence of Algorithm \ref{anfem} (see Theorem \ref{thm:conv2}). 
Hence, the convergence analysis motivates the two estimators in the \textrm{ESTIMATE} module of Algorithm \ref{anfem}. Furthermore, by leveraging the argument from adaptive conforming FEMs \cite{MR2832786}, the sufficient condition in Theorem \ref{thm:strong-conv} is proved to hold (see Lemma \ref{lem:max_est} and Theorem \ref{thm:conv_est}).

Now we discuss related works in existing studies. For linear or semi-linear (with nonlinearity in low-order terms) eigenvalue problems, there is a large body of literature on the linear Laplacian, diffusion, and bi-Laplacian  \cite{MR3647956,MR3532806,MR3315500,MR2810801,MR2970733,MR4727057,MR4700405,MR2823345,MR2787541,MR2430983,MR3259027,MR3347459,MR3407244,MR2832785,MR2531037,MR2485445}. Moreover, several works \cite{MR2911397,MR2240626,MR2383205,MR4067381,MR1950626,MR1935808} have investigated a posteriori error estimation for $p$-Laplacian boundary value problems. Despite these impressive advances, the AFEM approximation of the $p$-Laplacian eigenvalue problem \eqref{vp_eigen} remains very limited. Compared with boundary value problems, the eigenvalue problem itself is nonlinear even for the linear Laplacian, and hence \eqref{vp_eigen} associated with a nonlinear operator is naturally more computationally involved and theoretically challenging. The only existing work on problem \eqref{vp_eigen} is \cite{MR4860202}, which investigated the AFEM using conforming linear elements for the first eigenpair of \eqref{vp_eigen}. For a sufficiently fine initial mesh, Li et al. \cite{MR4860202} proved that the sequence of the first eigenvalues generated by the adaptive algorithm converges to $\lambda_1$ from above and that the $W^{1,p}$ distance between the sequence of associated discrete eigenfunctions and $E_{\lambda_1}$ tends to zero. The present work presents an extension of the work \cite{MR4860202} on the conforming method to a nonconforming one, aiming at approximating the first eigenpair from below. The extension is highly nontrivial due to the nonconformity of C-R FEs, and requires several new technical tools, especially discrete compactness on a sequence of discrete solutions from adaptive meshes. Note that the approximate first eigenpair from \cite{MR4860202} converges to the exact one from above due to its conformity, while the current approach with C-R FEs can provide a lower bound. 

The rest of the paper is structured as follows. In Section \ref{sec:nonconform-fem}, we describe the C-R FEM for problem \eqref{min_cont}. In Section \ref{sec:sufficient-cond}, we establish a sufficient condition for the convergence of discrete solutions using the AFEM, and in Section \ref{sec:afem}, we propose an adaptive algorithm and give its convergence analysis. In Section \ref{sec:num}, we present numerical results to illustrate the efficiency of the algorithm. Finally, in Section \ref{sec:conclusion}, we conclude with some remarks. In the appendix, we state several technical results on an $L^p$-based Sobolev space related to the divergence operator. Throughout, we employ standard notation for $L^p$ or $L^\infty$ spaces, the Sobolev space $W_0^{1,p}(\Omega)$, and their associated (semi-) norms.
We denote the conjugate exponent of $p$ by $q$, that is, $\frac{1}{q}+\frac{1}{p}=1$. Additionally, we use $C$, with or without a subscript, to denote a generic positive constant that is independent of the mesh size but its value may vary at each occurrence.

\section{Crouzeix-Raviart FEM}\label{sec:nonconform-fem}
Now we describe the C-R FE space for approximating problem \eqref{min_cont}. Let $\mathcal{T}$ be a {regular} conforming
triangulation of $\overline{\Omega}$ into a set of closed simplices with a discretization parameter $h_T:=|T|^{1/d}$, equivalent to $\mathrm{diam}(T)$, for each $T\in\mathcal{T}$. The set $\mathcal{F}_\mathcal{T}$ (respectively
$\mathcal{F}_{\mathcal{T}}(\Omega)$) consists of all faces (respectively all interior faces) in
$\mathcal{T}$. The local mesh size of each $F\in\mathcal{F}_{\mathcal{T}}$ is given by $h_F:=|F|^{1/(d-1)}$. Let $\mathcal{N}_{\mathcal{T}}$ denote the set of vertices of $\mathcal{T}$, and let $\mathcal{N}_{\mathcal{T}}(\Omega)$ (resp. $\mathcal{N}_{\mathcal{T}}({\partial \Omega})$) denote the interior vertices (resp. the boundary vertices). For $T\in\mathcal{T}$, 
$\boldsymbol{n}_T$ denotes the unit normal of $\partial T$ pointing towards the outside of $T$. We assign a fixed unit normal vector $\boldsymbol{n}_F$ to each
$F\in\mathcal{F}_{\mathcal{T}}$ and when $F\subset\partial\Omega$, $\boldsymbol{n}_F$ coincides with the unit outward normal $\boldsymbol{n}$ to $\partial\Om$. For $F\in \mathcal{F}_{\mathcal{T}}(\Omega)$ shared by two elements $T_1,T_2\in \mathcal{T}$ and $w_j=w|_{T_j}$ with $j=1,2$, we define the jump of $w$ on $F$ by $[w]:=\sum_{i=1}^2w_i(\boldsymbol{n}_{T_i}|_F, \boldsymbol{n}_F)$ as the jump across an interior face $F$ with $(\cdot,\cdot)$ the inner product in $\mathbb{R}^d$. For a boundary face $F\subset\partial\Omega$, $[w]:=w$.
Then C-R FE space over $\mathcal{T}$ is given by \cite{MR343661}:
\begin{equation}\label{CRspace}
    V_{\mathcal{T}}^{\op{CR}}:=\Big\{v\in L^p(\Om):~v|_{T}\in P_1(T), \forall T\in\mathcal{T},~\int_{F}[v]\ds=0, \forall F\in\mathcal{F}_{\mathcal{T}}\Big\}.
\end{equation}
A function in
$V_{\mathcal{T}}^{\op{CR}}$ is completely determined by its nodal values at the centers of all
interelement faces and it takes zero value at the centers of boundary faces. Over
$V_{\mathcal{T}}^{\op{CR}}$, we define a mesh-dependent broken norm:
\begin{align*}
\|\cdot\|_{1,p,\mathcal{T}}:=\Big(\sum_{T\in\mathcal{T}}\int_T|\boldsymbol{\nabla}\cdot|^p\dx\Big)^{1/p}.
\end{align*}
It reduces to the usual norm $\|\bold{\nabla}\cdot\|_{L^p(\Om)}$ over the $W_0^{1,p}(\Omega)$, by the Poincar\'{e} inequality. Moreover, there holds the discrete Poincar\'{e} inequality \cite[Section 10.6]{MR2373954} \cite[Corollary 4.3]{MR2557047} over $V_{\mathcal{T}}^{\op{CR}}$:
\begin{equation}\label{dis_poin}
    \|v\|_{L^p(\Om)}\leq C_{\mathrm{dp}}\|v\|_{1,p,\mathcal{T}},\quad\forall v\in V_{\mathcal{T}}^{\op{CR}},
\end{equation}
with the positive constant $C_{\mathrm{dp}}$ only depending on the shape regularity of $\mathcal{T}$.

Next, we define the energy functional over $V_{\mathcal{T}}^{\op{CR}}$:
\begin{align*}
\mathcal{J}_{\T}(v):=\sum_{T\in\T}\int_T|\boldsymbol{\nabla}v|^p\dx\Big/\int_\Om|v|^p\dx.
\end{align*}
Then the nonconforming FE approximation of \eqref{min_cont} reads: find $u_\mathcal{T}\in V_\mathcal{T}^{\op{CR}}$ such that
\begin{equation}\label{min_disc}
    \mu_\T:=\mathcal{J}_{\T}(u_\mathcal{T})=\inf_{v\in V_\mathcal{T}^{\op{CR}}\setminus \{0\}}\mathcal{J}_{\mathcal{T}}(v).
\end{equation}

\begin{theorem}\label{thm_min_disc}
The infimum in \eqref{min_disc} is positive and attained by some nonnegative $u_\mathcal{T}\in V_\mathcal{T}^{\op{CR}}\setminus \{0\}$.
\end{theorem}
\begin{proof}
We follow the argument in \cite{MR2722059} for the continuous problem \eqref{min_cont}. The discrete Poincar\'{e} inequality \eqref{dis_poin} implies
\begin{align*}
\mathcal{J}_{\T}(v_\T)=\sum_{T\in\T}\int_T|\boldsymbol{\nabla}v_\T|^p\dx\Big/\int_\Om|v_\T|^p\dx\geq C_{\mathrm{dp}}^{-p}>0,\quad\forall v_{\T}\in
     V_\mathcal{T}^{\op{CR}}\setminus\{0\}.
\end{align*}
Therefore, $\mu_\T$ is positive. Let $\{v_m\}_{m\geq 1}\subset
V_\mathcal{T}^{\op{CR}}\setminus\{0\}$ be a minimizing sequence. Since $\{|v_m|\}_{m\geq0}$ is also a
minimizing sequence, we may assume that $v_m\geq 0$ a.e. in $\Om$ for
all $m$. Moreover, the homogeneity of $\mathcal{J}_\T$ allows the normalization
$\int_\Om|v_m|^p\dx=1$. Then $\{\|v_m\|_{1,p,\mathcal{T}}\}_{m\geq1}$ is bounded in the finite dimensional space
$V_\mathcal{T}^{\op{CR}}$. There exist a subsequence, still denoted by $\{v_m\}_{m\geq1}$, and some $u_{\mathcal{T}}\in V_\mathcal{T}^{\op{CR}}$ such that
\[   \|v_m-u_\T\|_{1,p,\T}+\|v_m-u_\T\|_{L^p(\Om)}\to 0,\quad
    v_m\to u_\T~\mbox{a.e. in}~\Om.
\]
This implies $u_\mathcal{T}\geq0$ a.e. in $\Om$, $\int_\Om|u_\T|^p\dx=1$ and
$\mathcal{J}_\T(u_\T)=\lim_{m\to\infty}\mathcal{J}_\T(v_m)=\mu_{\T}$.
\end{proof}

The proof shows that the minimizer $u_\T$ to problem \eqref{min_disc} may be normalized as $\|u_\T\|_{L^p(\Omega)}=1$ in the analysis and the adaptive algorithm.
The differential calculus for $\mathcal{J}_\T$ indicates that the solution pair
$(\mu_\T,u_{\T})$ given by problem \eqref{min_disc} satisfies the discrete
formulation of problem \eqref{vp_eigen}:
\begin{equation}\label{vp_disc}
    \sum_{T\in\T}\int_T|\boldsymbol{\nabla}u_\T|^{p-2}\boldsymbol{\nabla}u_\T\cdot\boldsymbol{\nabla}v\dx=\mu_\T\int_\Omega|u_\T|^{p-2}u_\T v\dx,\quad\forall v\in V_\T^{\op{CR}}.
\end{equation}

\section{Sufficient condition for convergence}\label{sec:sufficient-cond}
In this section, we derive computable quantities that guarantee the convergence of nonconforming approximations generated by problem \eqref{min_disc}. More precisely, we prove that if these quantities vanish asymptotically and the initial mesh $\mathcal{T}_0$ is sufficiently refined, the resulting sequence of discrete solutions converges to the solution of \eqref{min_cont}.

Let $\mathbb{T}$ be the set of all possible conforming triangulations of
$\overline{\Omega}$ generated from an initial mesh $\mathcal{T}_0$ through successive bisection refinement \cite{MR1329875,MR2648380,MR2353951}. This refinement process guarantees uniform shape regularity in $\mathbb{T}$ (i.e., the shape regularity of any $\mathcal{T}\in\mathbb{T}$ is uniformly bounded by a constant depending on the initial mesh $\mathcal{T}_0$) \cite{MR2648380,MR1475530}. Thus all subsequent constants depend exclusively on $\mathcal{T}_0$ and problem data. We call $\mathcal{T}'$ a refinement of
$\mathcal{T}$ for any $\mathcal{T}$ and $\mathcal{T}'\in\mathbb{T}$ if $\mathcal{T}'$ is
produced from $\mathcal{T}$ by a finite number of bisections. For $T\in\mathcal{T}$, we denote by $D_{T}$ the patch of elements sharing at least one vertex with $T$, i.e., $D_T:=\bigcup\{T'\in\mathcal{T}: T'\cap T \neq\emptyset\}$. The nonconforming nature of
$V_{\mathcal{T}}^{\op{CR}}$ results in potential discontinuity across $F\in\mathcal{F}_{\mathcal{T}}(\Om)$, and therefore weak derivatives may not be square integrable. To address this issue, we employ the notation $\boldsymbol{\nabla}_{\mathcal{T}}$ to represent the piecewise gradient operator over $\mathcal{T}$. Then the broken Sobolev norm can be expressed as $\|\cdot\|^p_{1,p,\mathcal{T}}=\int_{\Om}|\boldsymbol{\nabla}_\mathcal{T}\cdot|^p \dx$. The dependence on $\mathcal{T}_k$ is indicated by the subscript $k$, e.g., $V_{k}^{\op{CR}}:=V_{\mathcal{T}_k}^{\op{CR}}$.

Let $\{(\mu_k,u_k)\}_{k\geq 0}$ be a sequence of discrete solutions to problem \eqref{min_disc} associated with a sequence of meshes $\{\mathcal{T}_{k}\}_{k\geq 0}\subset\mathbb{T}$ with $\mathcal{T}_{k+1}$ being a refinement of
$\mathcal{T}_{k}$, that is, the sequence $\{\mathcal{T}_{k}\}_{k\geq 0}\subset\mathbb{T}$ is nested. The discrete pair $\{(\mu_k,u_k)\}\in \mathbb{R}\times V_k^{\op{CR}}$ satisfies
\begin{equation}\label{vp_disc-adaptive}
\int_\Om|\boldsymbol{\nabla}_ku_k|^{p-2}\boldsymbol{\nabla}_ku_k\cdot\boldsymbol{\nabla}_kv\dx=\mu_k\int_\Omega|u_k|^{p-2}u_k v\dx,\quad\forall v\in V_k^{\op{CR}}.
\end{equation}
Next, we show that the sequence of eigenvalues $\{\mu_k\}_{k\geq 0}$ is bounded from above. Over each $\mathcal{T}_k$, we define the standard $W_0^{1,p}(\Om)$-conforming linear FE space $V_k$ by
\begin{align*}
V_k:=\{v\in V|~v|_{T}\in P_{1}(T),~\forall T\in\mathcal{T}_k\},
\end{align*}
and consider the following minimization problem: find
$\widetilde{u}_k\in V_k$ such that
\begin{equation}\label{min_dis_con}
\mathcal{J}_k(\widetilde{u}_k):=\inf_{v\in V_k\setminus\{0\}}\mathcal{J}_k(v).
\end{equation}
By the argument for Theorem \ref{thm_min_disc}, problem \eqref{min_dis_con} is well-posed and one can take $\|\widetilde{u}_k\|_{L^p(\Om)}=1$. Since $\{\mathcal{T}_k\}_{k\geq 0}$ is nested, $\{\mathcal{J}_k(\widetilde{u}_k)\}_{k\geq0}$ is a
monotonically decreasing sequence of positive numbers. This and the fact that $\mu_k=\mathcal{J}_{k}(u_k)=\|u_k\|_{1,p,k}^p\leq\mathcal{J}_{k}(\widetilde{u}_k)=|\widetilde{u}_k|_{1,p}^p$ (due to $V_k\subsetneq V_k^{\op{CR}}$) imply
\begin{equation}\label{disc_bd}
    \mu_k^{1/p} = \|u_k\|_{1,p,k}\leq |\widetilde{u}_k|_{1,p}\leq |\widetilde{u}_0|_{1,p}= C,\quad\forall k \geq 0.
\end{equation}

\begin{lemma}\label{lem:weakconv}
Let the sequence $\{\mathcal{T}_{k}\}_{k\geq 0}\subset\mathbb{T}$ be nested, and let $\{(\mu_k,u_k)\}_{k\geq 0}$ be its corresponding sequence of discrete solutions to problem \eqref{min_disc}. Suppose that
    \begin{equation}\label{con1_conv}
        \sum_{F\in\mathcal{F}_{k}}h_{F}\|[u_k]\|_{L^{p}(F)}^{p}\rightarrow0\quad \mbox{as~}k\to\infty.
    \end{equation}
Then there exists a subsequence $\{u_{k_j}\}_{j\geq0}$ and a limit $w\in W_0^{1,p}(\Omega)$ such that
        \begin{align}\label{wcon1}
            u_{k_j}\rightarrow& w\qquad\mbox{weakly in}~L^p(\Omega),\\
        \label{wcon2}
            \boldsymbol{\nabla}_{k_j}u_{k_j}\rightarrow& \boldsymbol{\nabla}w\quad\mbox{weakly in}~L^p(\Omega)^d.
        \end{align}
\end{lemma}
\begin{proof}
The uniform boundedness of $\{u_k\}_{k\geq0}$ in \eqref{disc_bd} and the normalization $\|u_k\|_{L^p(\Om)}=1$ imply
$        \|u_k\|_{L^p(\Om)}+\|u_k\|_{1,p,k}\leq C$.
Hence, we can extract a subsequence, denoted by $\{u_{k_j}\}_{j\geq 0}$, {as well as some functions} $w\in L^p(\Omega)$ and $\boldsymbol\sigma\in L^p(\Omega)^d$ such that
    \begin{equation}\label{pf1_lem2}
        u_{k_j}\rightarrow w~\mbox{weakly in}~L^p(\Om),\quad \boldsymbol{\nabla}_{k_j}u_{k_j}\rightarrow \boldsymbol{\sigma}~\mbox{weakly in}~L^p(\Om)^d.
    \end{equation}
    It suffices to prove $w\in V$ with $\boldsymbol{\nabla}w=\boldsymbol{\sigma}$. By elementwise integration by parts, we obtain
    \begin{equation}\label{pf2_lem2}
        \int_\Om u_{k_{j}}\textrm{div}\boldsymbol{\varphi}\dx
        =-\sum_{T\in\mathcal{T}_{k_{j}}}\int_T\boldsymbol{\nabla}u_{k_{j}}\cdot\boldsymbol{\varphi}\dx
        +\sum_{F\in\mathcal{F}_{k_{j}}}\int_F[u_{k_{j}}]\boldsymbol{\varphi}\cdot\boldsymbol{n}_{F}\ds,\quad\forall\boldsymbol{\varphi}\in
        C^{\infty}(\overline{\Omega})^{d}.
    \end{equation}
    Let $\overline{\boldsymbol{\varphi}}_{F}:=\int_F\boldsymbol{\varphi}\ds/|F|$ (i.e., the integral mean of $\boldsymbol{\varphi}$ over $F\in\mathcal{F}_{k_j}$). Since $\int_F[u_{k_j}]\ds=0$, we obtain
    \begin{align*}
        \left|\sum_{F\in\mathcal{F}_{k_{j}}}\int_F[u_{k_{j}}]\boldsymbol{\varphi}\cdot\boldsymbol{n}_{F}\ds\right|
        &=\left|\sum_{F\in\mathcal{F}_{k_{j}}}\int_F[u_{k_{j}}](\boldsymbol{\varphi}\cdot\boldsymbol{n}_{F}-\overline{\boldsymbol{\varphi}}_{F}\cdot\boldsymbol{n}_{F})\ds\right|\\
        &\leq\sum_{F\in\mathcal{F}_{k_{j}}}\int_F|[u_{k_{j}}]|\ds\|\boldsymbol{\varphi}-\overline{\boldsymbol{\varphi}}_F\|_{L^{\infty}(F)}.
        \end{align*}
By the interpolation error estimate and H\"{o}lder's inequality, we derive
    \begin{align*}
        \left|\sum_{F\in\mathcal{F}_{k_{j}}}\int_F[u_{k_{j}}]\boldsymbol{\varphi}\cdot\boldsymbol{n}_{F}\ds\right|
        &\leq C\left(\sum_{F\in\mathcal{F}_{k_j}}h_{F}\int_F|[u_{k_j}]|\ds\right)\|\boldsymbol{\nabla}\boldsymbol{\varphi}\|_{L^\infty(\Om)}\\
        &\leq C\left(\sum_{F\in\mathcal{F}_{k_j}}\int_Fh_{F}\ds\right)^{1/q}\left(\sum_{F\in\mathcal{F}_{k_j}}\int_F h_F|[u_{k_j}]|^{p}\ds\right)^{1/p}\|\boldsymbol{\nabla}\varphi\|_{L^\infty(\Om)}\\
        &\leq C|\Om|^{1/q}\left(\sum_{F\in\mathcal{F}_{k_j}}h_{F}\|[u_{k_j}]\|_{L^{p}(F)}^{p}\right)^{1/p}\|\boldsymbol{\nabla}\varphi\|_{L^\infty(\Om)}.
    \end{align*}
    Then inserting \eqref{pf1_lem2} and \eqref{con1_conv} into \eqref{pf2_lem2} leads to
    \[
        \int_\Om w\textrm{div}\boldsymbol{\varphi}\dx=-\int_\Om \boldsymbol{\sigma}\cdot\boldsymbol{\varphi}\dx,\quad\forall\boldsymbol{\varphi}\in C^\infty(\overline{\Om})^d.
    \]
This and the definition of distributional gradient yields $w\in W^{1,p}(\Om)$,
    with $\boldsymbol{\sigma}=\boldsymbol{\nabla}w$ in $L^p(\Omega)^d$. Since $C^{\infty}(\overline{\Omega})^{d}$ is dense in the space
    $\bold{W}^{q}(\mathrm{div};\Omega):=\{\boldsymbol{v}\in
    L^{q}(\Omega)^{d},\mathrm{div}\boldsymbol{v}\in L^{q}(\Omega)\}$ (see Theorem \ref{thm:density}), we further derive
    \[
        \int_\Om w\textrm{div}\boldsymbol{\varphi}\dx=-\int_\Om \boldsymbol{\nabla}w\cdot\boldsymbol{\varphi}\dx,\quad\forall\boldsymbol{\varphi}\in
        \bold{W}^{q}(\mathrm{div};\Omega).
    \]
Note that the normal-component trace space of $\bold{W}^{q}(\mathrm{div};\Omega)$ is
    $W^{-1/q,q}(\partial\Om)$ due to Theorems \ref{thm:trace} and \ref{thm:surjective} in the appendix, then an application of the Green's formula
    \eqref{Greenform} implies duality pairing $\langle \boldsymbol{\varphi}\cdot\boldsymbol{n},w\rangle_{W^{-1/q,q}(\partial\Om), W^{1/q,p}(\p\Om)}=0$ for any $\boldsymbol{\varphi}\in\bold{W}^{q}(\mathrm{div};\Omega)$, i.e., $w=0$ on $\partial\Om$.
\end{proof}

Since $\mu_{k}=\|u_{k}\|_{1,p,{k}}^p\leq C$ for any $k\geq 0$ by \eqref{disc_bd}, there exists a subsequence, denoted by $\{\mu_{k_j}\}_{j\geq0}$, and some $\mu\in\mathbb{R}$ such that $\mu_{k_j}\to\mu$. In view of Lemma \ref{lem:weakconv}, we shall prove that $(\mu,w)$ is an eigenpair of problem \eqref{vp_eigen}. Since $\|u_{k}\|_{L^p(\Omega)}=1$, it suffices to upgrade the weak convergence in \eqref{wcon1} to a strong one. Using the $W_0^{1,p}$-conforming linear FE space $V_k$, we define a connection operator \cite{MR1974504,MR2899264} $E_{k}:~V_{k}^{\op{CR}}\rightarrow V_{k}$ by
\begin{equation}\label{connop}
    E_{k} v(z)=\left\{\begin{array}{ll}
        \dfrac{1}{\#{N}(z)}\displaystyle{\sum_{T\in {N}(z)}v|_{T}(z)}&\quad z\in\mathcal{N}_k({\Omega}),\\
        0&\quad z\in\mathcal{N}_k({\partial\Omega}),
    \end{array}\right.
\end{equation}
where ${N}(z):=\{T\in\mathcal{T}_k|\; z\in\partial
T\}$ is the set of elements sharing the vertex $z$ and $\#{N}(z)$ is the cardinality of ${N}(z)$.
\begin{lemma}\label{lem3}
Let $E_{k}:~V_{k}^{\op{CR}}\rightarrow V_{k}$ be defined in \eqref{connop}. Then there holds
    \begin{align}\label{est_cop}
    \|v-E_kv\|_{L^{p}(\Omega)}^{p}\leq C\sum_{F\in\mathcal{F}_{k}}h_F\|[v]\|_{L^{p}(F)}^{p},\quad\forall v\in V_k^{\op{CR}},
    \end{align}
where the constant $C$ depends only on the shape regularity of $\mathcal{T}_k$.
\end{lemma}

\begin{proof}
For any $T\in\mathcal{T}_k$, let $z\in\mathcal{N}_k({\Omega})\cap T$ be an interior vertex.  Then for any $T'\in {N}(z)$, there exists a
sequence of elements $T_{1}=T,\ldots,T_{m}=T'$ in ${N}(z)$ with each consecutive pair
$T_{i}$ and $T_{i+1}$ sharing a common face, where $\#{N}(z)$ and $m$ are both bounded by a
constant depending on the shape regularity of $\mathcal{T}_k$. Then by the inverse estimate,
\begin{align*}
    |v|_{T}(z)-v|_{T'}(z)|^{p}\leq C\sum_{i=1}^{m-1}\left|v|_{T_i}(z)-v|_{T_i+1}(z)\right|^{p}
    \leq C\sum_{F\in\mathcal{V}(z)}h^{1-d}_{F}\|[v]\|_{L^{p}(F)}^{p},
\end{align*}
where $\mathcal{V}(z)$ is the set of faces $F$ with $z$ as a common vertex on $\partial F$. Then definition \eqref{connop} of $E_{k}$ implies
\begin{align}
    |v|_{T}(z)-E_{k}v(z)|^{p}
    &=\big|v|_{T}(z)-\frac{1}{\#{N}(z)}\sum_{T'\in {N}(z)}v|_{T'}(z)\big|^{p}\leq C\sum_{F\in\mathcal{V}(z)}h^{1-d}_{F}\|[v]\|_{L^{p}(F)}^{p}.\label{eq:111}
\end{align}
Similarly, for any boundary vertex $z\in\mathcal{N}_k{(\partial \Omega)}\cap T$, with $z\in F\subset T\cap\partial\Omega$, we have
\begin{align}
    |v|_{T}(z)-E_{k}v(z)|^{p}&=|v|_{T}(z)|^{p}
    \leq Ch^{1-d}_{F}\|[v]\|_{L^{p}(F)}^{p}.\label{eq:222}
\end{align}
On any $T\in\mathcal{T}_k$, the norm equivalence and the local quasi-uniformity of $\mathcal{T}_k$ imply
    \begin{align*}
       & \|v-E_{k}v\|_{L^{p}(T)}^{p}
        \leq C h_{T}^{d}\sum_{z\in\mathcal{N}_k\cap T}|v(z)-E_{k}v(z)|^{p},
    \end{align*}
which, together with \eqref{eq:111} and \eqref{eq:222}, yields
\begin{align}
        \|v-E_{k}v\|_{L^{p}(T)}^{p}
        &\leq Ch_T^d\Big(\sum_{z\in\mathcal{N}_k(\Omega)\cap T}\sum_{F\in\mathcal{V}(z)}h_{F}^{1-d}\|[v]\|_{L^{p}(F)}^{p}
        +\sum_{z\in\mathcal{N}_k(\partial \Omega)\cap T}\sum_{F\subset\p T}h_{F}^{1-d}\|[v]\|_{L^{p}(F)}^{p}\Big)\nonumber\\
        &\leq C\Big( \sum_{z\in\mathcal{N}_k(\Omega)\cap T}\sum_{F\in\mathcal{V}(z)}h_{F}\|[v]\|_{L^{p}(F)}^{p}
        +\sum_{z\in\mathcal{N}_k (\partial \Omega)\cap T}\sum_{F\subset\partial T}h_{F}\|[v]\|_{L^{p}(F)}^{p}\Big).\label{pf1_lem3}
    \end{align}
    Then by summing up \eqref{pf1_lem3} over all $T\in\mathcal{T}_k$, we obtain the desired estimate.
\end{proof}

\begin{remark}\label{rem:connop}
    The proof of Lemma \ref{lem3} is similar to that in \cite[Lemma 3.2]{MR1974504}, \cite[Lemma 3.2]{MR4191211} with $p=2$ and \cite[Lemma 10.6.6]{MR2373954} (for piecewise constant functions). By adapting the argument, we can also derive
    \begin{align*}
        \|v-E_kv\|_{L^{p}(T)}^{p}\leq C h_T^p \|\bold{\nabla}_k v\|^p_{L^p(D_T)}, \quad \forall v\in V_k^{\op{CR}}.
    \end{align*}
    This, the inverse estimate and the finite overlapping property of $\mathcal{T}_k$ yield the stability of $E_k$:
    \begin{equation}\label{connop_L^p_est}
        | E_k v |_{1,p} \leq C \| v \|_{1,p,k},\quad \forall v\in V_k^{\op{CR}},
    \end{equation}
    with the positive constant $C$ depending only on the shape regularity of $\mathcal{T}_k$.
\end{remark}

\begin{lemma}\label{lem:L^p-strongconv}
    Let condition \eqref{con1_conv} hold. Then the $L^p$ strong convergence holds for the weakly convergent subsequence $\{u_{k_j}\}_{j\geq0}$ in Lemma \ref{lem:weakconv}, i.e.
    \begin{equation}\label{L^p-strongconv}
     \lim_{j\to \infty}   \| u_{k_j} - w \|_{L^p(\Omega)}= 0.
    \end{equation}
\end{lemma}

\begin{proof}
The proof relies on the connection operator $E_k$ defined in \eqref{connop}. The Poincar\'{e} inequality, \eqref{connop_L^p_est} and \eqref{disc_bd} imply that $\{\|E_{k_j}u_{k_j}\|_{W^{1,p}(\Omega)}\}_{j\geq 0}$ is uniformly bounded. By Sobolev compact embedding theorem, there exist a Cauchy subsequence, still denoted by $\{E_{k_j}u_{k_j}\}_{j\geq 0}$, and some $\widehat{w}\in V$, satisfying
    \begin{equation}\label{lem:L^p-strongconv_pf1}
        \lim_{j\to\infty} \| E_{k_j}u_{k_j} - \widehat{w} \|_{L^p(\Omega)} = 0.
    \end{equation}
    Hence, Lemma \ref{lem3} and condition \eqref{con1_conv} lead to
    \begin{equation}\label{lem:L^p-strongconv_pf2}
        \| u_{k_j} - E_{k_j}u_{k_j} \|_{L^p(\Omega)} \leq C \left(\sum_{F\in\mathcal{F}_{k_j}}h_F\|[u_{k_j}]\|_{L^{p}(F)}^{p}\right)^{1/p}  \to 0 \quad \text{as}~j\to\infty.
    \end{equation}
Finally, by combining the estimates \eqref{lem:L^p-strongconv_pf1} and \eqref{lem:L^p-strongconv_pf2}, we derive
\begin{align*}
\lim_{j\to\infty} \|  u_{k_j} - \widehat{w}\|_{L^p(\Omega)} = 0.
\end{align*}
The uniqueness of the weak limit in Lemma \ref{lem:weakconv} implies $\widehat{w}=w$, and thus the assertion \eqref{L^p-strongconv} follows.
\end{proof}

In Lemmas \ref{lem:weakconv}-\ref{lem:L^p-strongconv}, the computable
quantity $\sum_{F\in\mathcal{F}_{k}}h_{F}\|[u_k]\|_{L^{p}(F)}^{p}$ measures the nonconformity of 
discrete solutions since $V_k^{\op{CR}} \nsubseteqq V$, and plays a very important role. 
The discrete compactness of C-R eigenfunctions $\{u_k\}_{k\geq 0}$ over locally refined meshes 
is recovered in Lemma \ref{lem:L^p-strongconv} under the condition that $\{\sum_{F\in\mathcal{F}_{k}}h_{F}\|[u_k]\|_{L^{p}(F)}^{p}\}_{k \geq 0}$ is a null sequence. The compactness property of nonconforming FE spaces on uniformly refined meshes was 
first established in \cite{MR566091}. Our result can be viewed as a version of the result for the adaptive mesh refinement strategy with $\sum_{F\in\mathcal{F}_{k}}h_{F}\|[u_k]\|_{L^{p}(F)}^{p}\to 0$ replacing the vanishing mesh size. 

Next we relate the limit pair $(\mu,w)$ to an eigenpair of the continuous problem \eqref{vp_eigen}. This involves an error indicator associated with problem \eqref{vp_disc-adaptive}. First, we define a linear functional associated with $\{(\mu_k,u_k)\}_{k\geq0}$:
\begin{align*}
    \langle\mathcal{R}(\mu_k,u_k),v\rangle:=
    \int_{\Omega}|\bold{\nabla}_{k}u_k|^{p-2}\boldsymbol{\nabla}_ku_{k}\cdot\boldsymbol{\nabla}v\dx
    -\mu_k\int_{\Omega}|u_k|^{p-2}u_kv\dx,\quad\forall v\in V.
\end{align*}
We define an interpolation operator ${I}_{k}:V\rightarrow V_{k}^{\op{CR}}$ \cite{MR343661,MR2783229}, which satisfies
\begin{equation}\label{CR_inter}
    \int_{F}{I}_{k}v\ds=\int_{F}v\ds,\quad\forall F\in\mathcal{F}_{k},~\forall v\in V.
\end{equation}
This operator satisfies the following stability estimate and the first-order approximation property \cite[Lemma 2]{MR2783229}:
for any $T\in\mathcal{T}_k$ and any $v\in V$, there hold
\begin{equation}\label{est_approx}
    \|\bold{\nabla} {I}_k v \|_{L^p(T)} \leq \| \bold{\nabla} v\|_{L^p(T)}, \quad \|v-{I}_{k}v\|_{L^p(T)}\leq (1+2/d)\mathrm{diam}(T)\|\bold{\nabla}v\|_{L^p(T)}.
\end{equation}

Then we get a lower bound on the first eigenvalue $\lambda_1$ of problem \eqref{vp_eigen} in terms of $\mu_k$.

\begin{theorem}\label{thm:GLB}
Let the sequence $\{\mathcal{T}_{k}\}_{k\geq 0}\subset\mathbb{T}$ be nested, and let $\{(\mu_k,u_k)\}_{k\geq 0}$ be its associated sequence of discrete solutions
to problem \eqref{min_disc}. If $\displaystyle(1+d/2)\max_{T\in\mathcal{T}_{k}}\mathrm{diam}(T) \lambda_{1}^{1/p}<1$, then there holds
    \begin{equation}\label{GLB}
        \frac{\mu_k^{1/p}}{\displaystyle \left( 1 + (1+d/2)\max_ {T\in\mathcal{T}_k}\mathrm{diam}(T)\mu_k^{1/p}\right)}  \leq  \lambda_{1}^{1/p}.
    \end{equation}
\end{theorem}
\begin{proof}
For any $u\in E_{\lambda_1}$, by taking $v_k = {I}_k u$ in \eqref{min_disc} over $\mathcal{T}_k$, using the stability 
estimate \eqref{est_approx} and assuming $\displaystyle(1+d/2)\max_{T\in\mathcal{T}_{k}}\mathrm{diam}(T) \lambda_{1}^{1/p}<1$, we get
\[
    \mu_{k} = \| \bold{\nabla}_k u_k \|_{L^p(\Omega)}^p \leq \frac{\| \bold{\nabla}_k {I}_k u \|_{L^p(\Omega)}^p }{\| {I}_k u \|_{L^p(\Omega)}^p } \leq \frac{\| \bold{\nabla} u \|_{L^p(\Omega)}^p }{\| {I}_k u \|_{L^p(\Omega)}^p } \leq \frac{\lambda_1 }{(1 - C_{\mathrm{apx}}\max_{T\in\mathcal{T}_k}\mathrm{diam}(T)\lambda_1^{1/p})^p},
\]
with $C_{\mathrm{apx}}=1+2/d$, from which, the desired lower bound follows.
\end{proof}

\begin{lemma}\label{lem:residue}
Let the sequence $\{\mathcal{T}_{k}\}_{k\geq 0}\subset\mathbb{T}$ be nested, and let $\{(\mu_k,u_k)\}_{k\geq 0}$ be its associated sequence of discrete solutions to problem \eqref{min_disc}. Suppose that
\begin{equation}\label{con2_conv}
    \mu_k^q\sum_{T\in\mathcal{T}_k}h_{T}^{q}\|u_k\|_{L^{p}(T)}^{p}\to 0~\quad\mbox{as}~k\to\infty.
\end{equation}
Then for the whole sequence $\{(\mu_k,u_k)\}_{k\geq0}$, there holds
\begin{equation}\label{con_res}
    \lim_{k\rightarrow\infty}\left\langle\mathcal{R}(\mu_k,u_k),v\right\rangle=0,\quad\forall v\in V.
\end{equation}
\end{lemma}

\begin{proof}
For any $v\in V$, we test \eqref{vp_disc-adaptive} with ${I}_{k}v\in V_{k}^{\op{CR}}$ and perform elementwise integration by parts, together with the identity $\mathrm{div}(|\bold{\nabla}_{k}u_{k}|^{p-2}\bold{\nabla}_k u_k) = 0$ on each $T\in\mathcal{T}_k$, then we derive
    \begin{align*}
        &\quad\left|\langle\mathcal{R}(\mu_k,u_k),v\rangle\right|\\
        &=
        \left|\int_\Omega|\boldsymbol{\nabla}_ku_k|^{p-2}\boldsymbol{\nabla}_ku_k\cdot\boldsymbol{\nabla}_k(v-{I}_kv)\dx
        -\mu_k\int_\Omega |u_k|^{p-2}u_k(v-{I}_kv)\dx
        \right|\\
        &=\left|\sum_{T\in\mathcal{T}_k}-\mu_k\int_T|u_k|^{p-2}u_k(v-{I}_kv)\dx
        +\sum_{T\in\mathcal{T}_{k}}\sum_{F\subset\partial T}\int_F|\boldsymbol{\nabla}_ku_k|^{p-2}\boldsymbol{\nabla}_ku_k\cdot\boldsymbol{n}_T(v-{I}_kv)\ds\right|.
    \end{align*}
Next, since $|\boldsymbol{\nabla}_ku_k|^{p-2}\boldsymbol{\nabla}_{k}{u_k}$ is a constant vector on $T\in\mathcal{T}_k$, the definition \eqref{CR_inter} directly implies
    \begin{align*}
        \sum_{T\in\mathcal{T}_{k}}\sum_{F\subset\partial T}\int_F|\boldsymbol{\nabla}_ku_k|^{p-2}\boldsymbol{\nabla}_ku_k\cdot\boldsymbol{n}_T(v-{I}_kv)\ds=0.
    \end{align*}
Note that $\mathrm{diam}(T)\leq C h_T$ due to the shape regularity of $\mathcal{T}_k$, then a combination of the identity $q=p/(p-1)$, H\"{o}lder's inequality and the approximation property in \eqref{est_approx}, leads to
    \begin{align*}
        \left|\langle\mathcal{R}(\mu_k,u_k),v\rangle\right|&\leq
        \sum_{T\in\mathcal{T}_k}\mu_k\left\||u_k|^{p-2} u_k\right\|_{L^{q}(T)}\|v-{I}_kv\|_{L^p(T)}\\
        &\leq C\Big(\sum_{T\in\mathcal{T}_k}h_T^{q}\left\|\mu_k |u_k|^{p-2}u_k\right\|_{L^{q}(T)}^{q}\Big)^{1/{q}}\|\bold{\nabla}v\|_{L^p(\Om)}.
    \end{align*}
Finally, condition \eqref{con2_conv} implies the desired assertion \eqref{con_res}.
\end{proof}

We define two nonlinear mappings $\mathcal{F}(\bold{\tau}):=|\bold{\tau}|^{p-2}\bold{\tau}$: $L^p(\Om)^d$ $\to$ $L^{q}(\Om)^d$ and $\mathcal{G}(v):=|v|^{p-2}v$: $L^p(\Om)$ $\to$ $L^{q}(\Om)$.

\begin{lemma}\label{lem:diff_func}
The mappings $\mathcal{F}$ and $\mathcal{G}$ are both continuous with respect to the $L^p$-norm.
\end{lemma}

\begin{proof}
We take any sequence $\{\boldsymbol{\tau}_n\}_{n\geq 1}\subset L^p(\Om)^d$ such that $\bold{\tau}_n\to\bold{\tau}$ strongly in $L^p(\Om)^d$ as $n\to\infty$. Then  by Fischer-Riesz Theorem (cf. \cite[Theorem 1.2.7]{MR2722059}, \cite[Theorem 4.9]{MR2759829}), there exist a subsequence {$\{\boldsymbol{\tau}_{n_j}\}_{j\geq 0}$} and $w_1\in L^1(\Om)$ satisfying
$\boldsymbol{\tau}_{n_j}\to\boldsymbol{\tau}$ a.e. in $\Om$ and $|\boldsymbol{\tau}_{n_j}|^p\leq w_1$ a.e. in $\Om$ for all $j$. This implies
\begin{align*}
    \left|\boldsymbol{\tau}_{n_j}\right|^{p-2}\boldsymbol{\tau}_{n_j}&\to|\boldsymbol{\tau}|^{p-2}\boldsymbol{\tau}\quad\mbox{a.e. in}~\Om\quad \text{as}~j\to\infty,\\
    \left||\boldsymbol{\tau}_{n_j}|^{p-2}\boldsymbol{\tau}_{n_j}-|\boldsymbol{\tau}|^{p-2}\boldsymbol{\tau}\right|^{q}&\leq (|\boldsymbol{\tau}_{n_j}|^{p-1}+|\boldsymbol{\tau}|^{p-1})^{q}\leq C(|\boldsymbol{\tau}_{n_j}|^p+|\boldsymbol{\tau}|^p)\leq C(w_1+|\boldsymbol{\tau}|^p)\in L^1(\Om).
\end{align*}
Then by Lebesgue's dominated convergence theorem, we obtain
\[
\int_{\Om}||\boldsymbol{\tau}_{n_j}|^{p-2}\boldsymbol{\tau}_{n_j}-|\boldsymbol{\tau}|^{p-2}\boldsymbol{\tau}|^{q}\dx\to 0\quad \text{as}~j\to\infty.
\]
Since for every sequence $\boldsymbol{\tau}_n\to\boldsymbol{\tau}$ in $L^p(\Om)^d$ as $n\to\infty$, there exists a subsequence {$\{\boldsymbol{\tau}_{n_j}\}_{j\geq 0}$} such that $|\boldsymbol{\tau}_{n_j}|^{p-2}\boldsymbol{\tau}_{n_j}\to|\boldsymbol{\tau}|^{p-2}\boldsymbol{\tau}$ in $L^{q}(\Om)^d$ as $j\to\infty$, the standard subsequence contradiction argument implies that the whole sequence also satisfies $|\boldsymbol{\tau}_{n}|^{p-2}\boldsymbol{\tau}_{n}\to|\boldsymbol{\tau}|^{p-2}\boldsymbol{\tau}$
in $L^{q}(\Om)^d$ as $n\to\infty$. This argument applies also to the strong continuity of $\mathcal{G}$ from $L^p(\Om)$ to $L^{q}(\Om)$.
\end{proof}

Now we can state the first main result of this section.

\begin{theorem}\label{thm:strong-conv}
Let the sequence $\{\mathcal{T}_{k}\}_{k\geq 0}\subset\mathbb{T}$ be nested, and let $\{(\mu_k,u_k)\}_{k\geq 0}$ be the associated sequence of discrete solutions to problem \eqref{min_disc}. Then under the conditions \eqref{con1_conv} and \eqref{con2_conv}, i.e.,
 \begin{equation}\label{condition 1}
         \mu_k^q\sum_{T\in\mathcal{T}_k}h_{T}^{q}\|u_{k}\|_{L^{p}(T)}^{p}+\sum_{F\in\mathcal{F}_k}h_F\|[u_k]\|^{p}_{L^{p}(F)}\to
        0\quad\mbox{as}~k\to\infty,
     \end{equation}
     there exist a subsequence $\{(\mu_{k_j},u_{k_j})\}_{j\geq0}$ and a pair $(\mu,w)\in \mathbb{R}\times V$ such that
     \begin{equation}\label{conv1_sol}
        \lim_{j\to\infty}\mu_{k_j} = \mu,\quad\lim_{j\to\infty}\|u_{k_j}-w\|_{L^p(\Om)}=0,\quad\lim_{j\to\infty}\|u_{k_j}-w\|_{1,p,k_j}=0.
     \end{equation}
\end{theorem}
\begin{proof}
First, the inequality \eqref{disc_bd} implies boundedness of the discrete eigenvalues:
\begin{align*}
\mu_k=\||\boldsymbol{\nabla}_{k}u_k|^{p-2}\boldsymbol{\nabla}_{k}u_k\|^{q}_{L^{q}(\Om)}=\|u_k\|^p_{1,p,k}\leq C.
\end{align*}
Then there exist two subsequences $\{\mu_{k_j}\}_{j\geq0}$ and $\{|\boldsymbol{\nabla}_{k_j}u_{k_j}|^{p-2}\boldsymbol{\nabla}_{k_j}u_{k_j}\}_{j\geq0}$ with limits $\mu\in\mathbb{R}$ and $\boldsymbol{\xi}\in L^{q}(\Om)^d$ such that
\begin{equation}\label{pf1_thm:conv1}
\mu_{k_j}\to\mu\quad \mbox{and}\quad |\boldsymbol{\nabla}_{k_j}u_{k_j}|^{p-2}\boldsymbol{\nabla}_{k_j}u_{k_j}\rightarrow \boldsymbol{\xi}\quad\mbox{weakly in}~L^{q}(\Omega)^d.
\end{equation}
Thus, the first convergence in \eqref{conv1_sol} follows.
Next, by Lemmas \ref{lem:weakconv} and \ref{lem:L^p-strongconv}, there exist a subsequence, still denoted by $\{u_{k_j}\}_{j\geq0}$, and $w\in W^{1,p}_0(\Om)$ such that
\begin{equation}\label{pf2_thm:conv1}
               u_{k_j}\rightarrow w\quad\mbox{strongly in}~L^p(\Omega),\quad
            \boldsymbol{\nabla}_{k_j}u_{k_j}\rightarrow \boldsymbol{\nabla}w\quad\mbox{weakly in}~L^p(\Omega)^d,
\end{equation}
i.e., the second convergence in \eqref{conv1_sol} holds.
Now \eqref{pf1_thm:conv1}, the $L^p$ strong convergence in \eqref{pf2_thm:conv1} and Lemma \ref{lem:diff_func} lead to
\[
    \lim_{j\to\infty}\langle\mathcal{R}(u_{k_j},\mu_{k_j}),v\rangle=
    \int_{\Om}\boldsymbol{\xi}\cdot\boldsymbol{\nabla}v\dx-\mu\int_{\Om}|w|^{p-2}wv\dx,\quad\forall v\in V.
\]
Condition \eqref{condition 1} and Lemma \ref{lem:residue} imply $\lim_{j\to\infty}\langle\mathcal{R}(u_{k_j},\mu_{k_j}),v\rangle=0$. Consequently, 
\begin{equation}\label{pf3_thm:conv1}
    \int_{\Om}\boldsymbol{\xi}\cdot\boldsymbol{\nabla}v\dx=\mu\int_{\Om}|w|^{p-2}wv\dx,\quad\forall v\in V.
\end{equation}
Since $\|u_{k_j}\|_{L^p(\Om)}=1$, together with the strong convergence in \eqref{pf2_thm:conv1}, there holds $\|w\|_{L^p(\Om)}=1$. Taking $v=w$ in \eqref{pf3_thm:conv1} gives
\begin{equation}\label{pf4_thm:conv1}
\int_{\Om}\boldsymbol{\xi}\cdot\boldsymbol{\nabla}w\dx=\mu.
\end{equation}
Note that $|\nabla w|^{p-2}\nabla w\in L^q(\Omega)^d$ due to $w\in W_0^{1,p}(\Omega)$, by the weak convergence in \eqref{pf2_thm:conv1}, we derive
\begin{align*}
    \lim_{j\to\infty}\int_\Om|\boldsymbol{\nabla}w|^{p-2}\boldsymbol{\nabla}w\cdot (\boldsymbol{\nabla}_{k_j}u_{k_j}-\boldsymbol{\nabla}w) \dx = 0.
\end{align*}
Moreover, by testing \eqref{vp_disc-adaptive} with $u_{k_j}\in V_{k_j}^{\op{CR}}$, the normalization $\|u_{k_j}\|_{L^p(\Om)}=1$ and \eqref{pf1_thm:conv1}, we obtain
\begin{align*}
\lim_{j\to\infty}\int_{\Om}|\boldsymbol{\nabla}_{k_j}u_{k_j}|^{p}\dx=\lim_{j\to\infty}\mu_{k_j}=\mu.
\end{align*}
Combining with the weak convergence in \eqref{pf1_thm:conv1} and \eqref{pf4_thm:conv1}, we derive
\begin{align*}
    \lim_{j\to\infty}\int_{\Om}|\boldsymbol{\nabla}_{k_j}u_{k_j}|^{p}\dx-\int_\Om|\boldsymbol{\nabla}_{k_j}u_{k_j}|^{p-2}\boldsymbol{\nabla}_{k_j}u_{k_j}\cdot\boldsymbol{\nabla}w \dx = \mu-\int_{\Om}\boldsymbol{\xi}\cdot\boldsymbol{\nabla}w\dx=0.
\end{align*}
Now by H\"{o}lder's inequality, we derive
\[
    \begin{aligned}
        &\quad\int_{\Om}(|\boldsymbol{\nabla}_{k_j}u_{k_j}|^{p-2}\boldsymbol{\nabla}_{k_j}u_{k_j}-|\boldsymbol{\nabla}w|^{p-2}\boldsymbol{\nabla}w)\cdot(\boldsymbol{\nabla}_{k_j}u_{k_j}-\boldsymbol{\nabla}w)\dx\\
        &=\int_{\Om}\left(|\boldsymbol{\nabla}_{k_j}u_{k_j}|^p+|\boldsymbol{\nabla}w|^p-|\boldsymbol{\nabla}_{k_j}u_{k_j}|^{p-2}\boldsymbol{\nabla}_{k_j}u_{k_j}\cdot\boldsymbol{\nabla}w-|\boldsymbol{\nabla}w|^{p-2}\boldsymbol{\nabla}w\cdot\boldsymbol{\nabla}_{k_j}u_{k_j}\right) \dx\\
        &\geq
        \int_{\Om}|\boldsymbol{\nabla}_{k_j}u_{k_j}|^{p} \dx + \int_{\Om}|\boldsymbol{\nabla}w|^{p} \dx-\left(\int_{\Om}|\boldsymbol{\nabla}_{k_j}u_{k_j}|^{p} \dx\right)^{\frac{p-1}{p}}\left(\int_{\Om}|\boldsymbol{\nabla}w|^p \dx\right)^{\frac{1}{p}}\\
        &\qquad-\left(\int_{\Om}|\boldsymbol{\nabla}w|^{p} \dx\right)^{\frac{p-1}{p}}\left(\int_{\Om}|\boldsymbol{\nabla}_{k_j}u_{k_j}|^p \dx\right)^{\frac{1}{p}}\\
        &=\left(\|u_{k_j}\|_{1,p,k_j}^{p-1}-\|\boldsymbol{\nabla}w\|_{L^p(\Om)}^{p-1}\right)\left(\|u_{k_j}\|_{1,p,k_j}-\|\boldsymbol{\nabla}w\|_{L^p(\Om)}\right)\geq 0.
    \end{aligned}
\]
This and the two vanishing limits yield
$\|u_{k_j}\|_{1,p,k_j}\to\|\boldsymbol{\nabla}w\|_{L^p(\Om)}$ as $j\to\infty$.
Together with the weak convergence in \eqref{pf2_thm:conv1} and the uniform convexity of $L^p(\Om)^d$, this gives the third result in \eqref{conv1_sol}.
\end{proof}

\begin{remark}\label{rem:conv}
    The proof of Theorem \ref{thm:strong-conv} indicates that for any subsequence $\{(\mu_{k_j},u_{k_j})\}_{j\geq0}$ of $\{(\mu_k,u_k)\}_{k\geq0}$, we can get another subsequence $\{(\mu_{k_{j_m}},u_{k_{j_m}})\}_{m\geq0}$ and some pair $(\widetilde{\mu},\widetilde{w})\in\mathbb{R}\times V$ such that $\|u_{k_{j_{m}}}-\widetilde{w}\|_{L^p(\Om)}\to 0$, $\|u_{k_{j_m}}-\widetilde{w}\|_{1,p,k_{j_m}}\to 0$ and $\mu_{k_{j_m}}\to \widetilde{\mu}$.
\end{remark}

\begin{theorem}\label{thm:conv2}
Let the sequence $\{\mathcal{T}_{k}\}_{k\geq 0}\subset\mathbb{T}$ be nested, and let $\{(\mu_k,u_k)\}_{k\geq 0}$ be the associated sequence of discrete solutions to problem \eqref{min_disc}.
Suppose that the initial mesh $\mathcal{T}_0$ is sufficiently fine and condition \eqref{condition 1} is satisfied. Then there holds
     \begin{align}\label{conv_eigen}
        \lim_{k\to\infty}\mu_{k}=\lambda_1\quad \mbox{and}\quad\lim_{k\to\infty}\inf_{v\in E_{\lambda_1}}\|u_k-v\|_{1,p,k}=0.
     \end{align}
\end{theorem}

\begin{proof}
The proof is divided into four steps.

\noindent\textbf{Step 1.} Let $(\mu,w)\in\mathbb{R}\times V$ be defined in \eqref{conv1_sol}. Then Lemmas \ref{lem:residue} and \ref{lem:diff_func} and Theorem \ref{thm:strong-conv} imply
\begin{equation}\label{pf1_thm:conv2}
    \int_{\Om}|\bold{\nabla}w|^{p-2}\bold{\nabla}w\cdot\bold{\nabla}v\dx
    =\mu\int_{\Om}|w|^{p-2}wv\dx\quad\forall~v\in V,
\end{equation}
i.e. $(\mu,w)$ is an eigenpair of problem \eqref{vp_eigen} with $\|w\|_{L^p(\Omega)}=1$.

\noindent\textbf{Step 2.} Let $E:=\{v\in V|~v~\mbox{satisfies \eqref{vp_eigen} with}~\|v\|_{L^p(\Omega)}=1~\text{for some}~\lambda\in \mathbb{R}\}$, which consists of all eigenfunctions of problem \eqref{vp_eigen}. Since problem \eqref{vp_eigen} has a nondecreasing sequence of positive eigenvalues \cite{MR912211,MR2196811,MR1007505}, there holds $E_{\lambda_1}\subsetneq E$. Since any $u\in E_{\lambda_1}$ is also a minimizer of $\mathcal{J}$ over $W_0^{1,p}(\Om)$, $\lambda_1=\mathcal{J}(u)\leq\inf_{v\in E\setminus
    E_{\lambda_1}}\mathcal{J}(v)$. Suppose that the equality is attained. Let
    $\{w_n\}_{n\geq0}\subset E\setminus E_{\lambda_1}$ be a minimizing sequence such that
    $\mathcal{J}(w_n)\to\inf_{v\in E\setminus
    E_{\lambda_1}}\mathcal{J}(v)=\lambda_1$. Then $\mathcal{J}(w_n)$ is also a sequence of eigenvalues other than $\lambda_1$. This contradicts the fact that $\lambda_1$ is isolated \cite{MR1007505}.  Hence $\lambda_1=\mathcal{J}(u)<\inf_{v\in E\setminus
    E_{\lambda_1}}\mathcal{J}(v)$ for any $u\in E_{\lambda_1}$.

\noindent\textbf{Step 3.} Recall that $V_{k}$ is the $W^{1,p}_{0}(\Om)$-conforming linear FE space over $\mathcal{T}_k$, and 
that $\widetilde{u}_k\in V_k$ is the minimizer defined in \eqref{min_dis_con}. Let $V_\infty:=\overline{\bigcup_{k\geq 0}V_{k}}$ 
in the $W^{1,p}$-norm. If $\{\mathcal{T}_k\}_{k\geq 0}$ is a sequence of uniformly refined meshes, then for any
    $u\in E_{\lambda_1}$ there exists a sequence $\{v_l\}_{l\geq0}\subset V_\infty$ such that $v_l\to u$ strongly in $W_0^{1,p}(\Om)$. Note that $\mathcal{J}$ is continuous over
    $W_0^{1,p}(\Om)$ and $\mathcal{J}(u)<\inf_{v\in E\setminus
    E_{\lambda_1}}\mathcal{J}(v)$ as proved in Step 2, we derive $\mathcal{J}(v_l)<\inf_{v\in E\setminus
    E_{\lambda_1}}\mathcal{J}(v)$ for sufficiently large $l$ or a sufficiently small mesh size of $\mathcal{T}_l$. With \eqref{min_dis_con} and the inclusive property $V_{k}\subsetneq V_{k}^{\op{CR}}$ invoked, a fine enough initial $\mathcal{T}_0$, in the sequence of $\{\mathcal{T}_k\}_{k\geq 0}$, guarantees that $\mu_k=\mathcal{J}_k(u_k)\leq\mathcal{J}_{k}(\widetilde{u}_k)\leq\mathcal{J}_0(\widetilde{u}_0)<\inf_{v\in E\setminus
    E_{\lambda_1}}\mathcal{J}(v)$. Taking $v=w$ in \eqref{pf1_thm:conv2} and noting $\mu_{{k}_j}\to\mu$ in \eqref{conv1_sol} as well as $\|w\|_{L^p(\Omega)}=1$, we get $\mathcal{J}(w)=\mu<\inf_{v\in E\setminus
    E_{\lambda_1}}\mathcal{J}(v)$, which, along with \eqref{pf1_thm:conv2} and $\|w\|_{L^p(\Omega)}=1$ again, implies $w\in E_{\lambda_1}$ and $\mu=\lambda_1$. Moreover, this argument shows that $\widetilde{\mu}=\lambda_1$ for any convergent subsequence of $\{(\mu_k,u_k)\}_{k\geq0}$ with a limit pair $(\widetilde{\mu},\widetilde{w})$. So the limit $\lambda_1$ is independent of a particular subsequence. This leads to the first result in \eqref{conv_eigen}.

\noindent\textbf{Step 4.} To prove the second result in \eqref{conv_eigen}, we also argue by contradiction. If the result is false, then there exist a number $\varepsilon>0$ and a subsequence $\{u_{k_j}\}_{j\geq0}$ of $\{u_k\}_{k\geq0}$ such that $\inf_{v\in E_{\lambda_1}}\|u_{k_j}-v\|_{1,p,k_j}\geq\varepsilon$ for all $k_j$. By Remark \ref{rem:conv}, we can extract a subsequence of discrete eigenfunctions $\{u_{k_{j_m}}\}_{m\geq0}$ converging to some $\widetilde{w}\in W_0^{1,p}(\Om)$ and a subsequence of corresponding eigenvalues $\{\mu_{k_{j_m}}\}_{m\geq0}$ to some $\widetilde{\mu}\in\mathbb{R}$. By repeating the argument in Steps 1-3, we deduce that $(\widetilde{\mu},\widetilde{w})$ satisfies \eqref{pf1_thm:conv2} with $\widetilde{\mu}=\lambda_1$ and $\|\widetilde{w}\|_{L^p(\Omega)}=1$, i.e., $\widetilde{w}\in E_{\lambda_1}$. This contradicts the assumption that $\{u_{k_j}\}_{j\geq0}$ does not converge to any eigenfunction in $E_{\lambda_{1}}$.
\end{proof}

\section{Adaptive algorithm}\label{sec:afem}
Theorem \ref{thm:conv2} implies that condition \eqref{condition 1} on the computable quantities (and a sufficiently fine initial mesh $\mathcal{T}_0$) ensures the convergence of relevant C-R FE approximations. Motivated by this observation, we use
the computable quantity to design an adaptive algorithm to approximate problem \eqref{min_cont} even if it is not a reliable bound of the error. To this end, we first introduce some notation. For any $(\nu, v)\in\mathbb{R}_{>0}\times V_k^{\op{CR}}$ and any $T\in\mathcal{T}_k$, we define two local error indicators
\[
    \eta_{k,1}(\nu,v,T):=\nu^q h_T^{q}\| v\|_{L^{p}(T)}^{p},\quad \eta_{k,2}(v,T):=\sum_{F\subset\partial T\cap\Omega}\frac{1}{2}h_{F}\|[v]\|_{L^{p}(F)}^{p}+\sum_{F\subset\partial T\cap\partial\Omega}h_{F}\|[v]\|_{L^{p}(F)}^{p},
\]
and two global error estimators 
\begin{align*} 
\eta_{k,1}(\nu,v,\mathcal{M}):=\sum_{T\in\mathcal{M}}\eta_{k,1}(\nu,v,T)\quad \mbox{and}\quad \eta_{k,2}(v,\mathcal{M}):=\sum_{T\in\mathcal{M}}\eta_{k,2}(v,T),
\end{align*}
for any $\mathcal{M}\subseteq\mathcal{T}_k$. If $\mathcal{M}=\mathcal{T}_k$, then it is dropped. 

\begin{algorithm}
\caption{Adaptive Crouzeix-Raviart FEM for the first eigenvalue of $p$-Laplacian}\label{anfem}
\begin{algorithmic}[1]
    \State {(INITIALIZE)} Specify an initial conforming mesh $\mathcal{T}_0$ and set $k:=0$.

    \State {(SOLVE)} Solve Problem \eqref{min_disc} on $\mathcal{T}_k$ for $(\mu_k,u_{k})\in \mathbb{R} \times V_{k}^{\op{CR}}$ s.t. $\|u_k\|_{L^p(\Omega)} = 1$.

    \State {(ESTIMATE)} Compute two error estimators $\eta_{k,1}(\mu_k,u_k)$ and $\eta_{k,2}(u_k)$.

    \State {(MARK)} Mark two subsets $\mathcal{M}_k^{i}\subseteq\mathcal{T}_k$ ($i=1,2$), each containing at least one element $T_k^{i}$ holding the largest error indicator {among all $\eta_{k,1}(\mu_k,u_k,T)$ or $\eta_{k,2}(u_k,T)$}, i.e.,
    \begin{equation}\label{marking}
   \eta_{k,1}(\mu_k,u_k,T_k^1)=\max_{T\in\mathcal{T}_k}\eta_{k,1}(\mu_k,u_k,T)\quad \mbox{and} \quad \eta_{k,2}(u_k,T_k^2)=\max_{T\in\mathcal{T}_k}\eta_{k,2}(u_k,T).
\end{equation}
                Then $\mathcal{M}_k:=\mathcal{M}_k^1\cup \mathcal{M}_k^2$.

    \State {(REFINE)} Refine each $T\in\mathcal{M}_k$ by bisection to get $\mathcal{T}_{k+1}$.

    \State Set $k:=k+1$ and go to Step 2.
    \end{algorithmic}
\end{algorithm}

The MARK module consists of two separate markings based on $\eta_{k,1}(\mu_k,u_k)$ and $\eta_{k,2}(u_k)$. The former measures the residual of equation \eqref{diff_eq} (with $(\lambda,u)$ replaced by $(\mu_k,u_k)$) on each element of $\mathcal{T}_k$, while the latter measures the discontinuity of $u_k$. Note that several commonly used marking strategies, e.g., maximum strategy, equi-distribution strategy, modified equi-distribution strategy and D\"{o}rfler's strategy, meet the requirement \eqref{marking} in the MARK module. The bisection procedure in the REFINE module follows rules in \cite{MR1329875,MR2648380,MR2353951}. Algorithm \ref{anfem} generates a sequence $\{\mathcal{T}_k\}_{k\geq0}\subset\mathbb{T}$. Let
\[
    \mathcal{T}_{k}^{+}:=\bigcap_{l\geq k}\mathcal{T}_{l},\quad
    \mathcal{T}_{k}^{0}:=\mathcal{T}_{k}\setminus\mathcal{T}_{k}^{+},\quad
    \Omega_{k}^{+}:=\bigcup_{T\in\mathcal{T}^{+}_{k}}D_{T},\quad
    \Omega_{k}^{0}:=\bigcup_{T\in\mathcal{T}^{0}_{k}}D_{T}.
\]
That is, $\mathcal{T}_{k}^{+}$ consists of all elements that are not refined after the $k$-th loop while
all elements in $\mathcal{T}_{k}^{0}$ are refined at least once after the $k$-th iteration. Clearly, $\mathcal{T}_{l}^{+}\subset\mathcal{T}_{k}^{+}$ for $l<k$ and $\mathcal{M}_k\subset\T_{k}^0$. We also define a mesh-size function
$h_{k}:\overline{\Omega}\rightarrow\mathbb{R}^{+}$ almost everywhere by $h_{k}(x)=h_{T}$ for $x$ in the interior of an element $T\in\mathcal{T}_{k}$ and
$h_{k}(x)=h_{F}$ for $x$ in the relative interior of a face $F\in\mathcal{F}_{k}$.
With $\chi^{0}_{k}$ being the characteristic function of $\Omega_{k}^{0}$, the mesh-size function satisfies \cite{MR2832786}:
\begin{equation}\label{con_0_mz}
    \lim_{k\rightarrow\infty}\|h_{k}\chi^{0}_{k}\|_{L^\infty(\Omega)}=0.
\end{equation}

The next lemma gives the vanishing limit of the error estimators.
\begin{lemma}\label{lem:max_est}
    Let $\{(\mu_k,u_k)\}_{k\geq 0}$ be the sequence of discrete eigenpairs and $\{\mathcal{M}_k\}_{k\geq0}$ the sequence of marked sets generated by Algorithm \ref{anfem}. Then there holds
    \begin{align}\label{con_0_max}
\lim_{k\to\infty}\max_{T\in\mathcal{M}_k}\eta_{k,1}(\mu_k,u_k,T)=\lim_{k\to\infty}\max_{T\in\mathcal{M}_k}\eta_{k,2}(u_k,T) = 0.
    \end{align}
\end{lemma}
\begin{proof}
In Algorithm \ref{anfem}, let $T_k^{i}\in\mathcal{M}_k$ ($i=1,2$) be the element with the largest error indicator $\eta_{k,1}(\mu_k, u_k, T_k^1)$ and $\eta_{k,2}(u_k, T_k^2)$ over $\mathcal{M}_k$, respectively. The bound \eqref{disc_bd} and $\|u_k\|_{L^p(\Omega)}=1$ give
\begin{equation}\label{pf1_max_est}
\eta_{k,1}(\mu_k,u_k,T_k^1)=h_{T^1_k}^{q}\mu_{k}^{q}\|u_k\|^{p}_{L^{p}(T^1_k)}
            \leq
        C h_{T^1_k}^{q}\|u_k\|_{L^p(\Om)}^{p}= C h_{T^1_k}^{q}.
\end{equation}
{Note that} $[u_k]=0$ at the center of $F$ and  $\boldsymbol{\nabla}_{k}u_{k}$ is a piecewise
    constant vector over $\mathcal{T}_k$, then a combination of the scaled trace theorem, 
    local quasi-uniformity of $\mathcal{T}_k$ and \eqref{disc_bd}, leads to
    \begin{equation}\label{pf2_max_est}
        \begin{aligned}
            \eta_{k,2}(u_k,T_k^2)\leq&\sum_{F\subset\partial T^2_k}h_{F}\int_F|[u_k]|^{p}\ds\leq C\sum_{F\subset\partial T^2_k}h_F^{p+1}\int_{F}|[\boldsymbol{\nabla}_k u_k\times\boldsymbol{n}_{F}]|^{p}\ds\\
            \leq&C h_{T^2_k}^{p}\int_{D_{T_{k}^2}}|\boldsymbol{\nabla}_k u_k|^{p}\dx
        \leq Ch_{T^2_k}^{p}\|u_k\|_{1,p,k}^{p}\leq Ch_{T^2_k}^{p} \quad d=3.
        \end{aligned}
    \end{equation}
    For $d=2$, let $\bold{t}_F$ be the unit tangent vector given by rotating $\bold{n}_F$ $90^{\circ}$ counter-clockwise. With $[\boldsymbol{\nabla}_k u_k\cdot\bold{t}_{F}]$ replacing $[\boldsymbol{\nabla}_k u_k\times\boldsymbol{n}_{F}]$ in \eqref{pf2_max_est}, a similar argument yields the estimate for $d=2$. Since $T^1_k$ and $T^2_k\in\mathcal{M}_k\subset\mathcal{T}_k^0$, then \eqref{con_0_mz} implies 
    \begin{align}\label{pf3_max_est}
        h_{T^1_k}, h_{T^2_k}\leq \|h_k\chi^0_k\|_{L^{\infty}(\Om)}\to 0\quad \text{as}~k\to\infty.
    \end{align}
This, \eqref{pf1_max_est}, and \eqref{pf2_max_est} give the desired assertion.
\end{proof}

Now we state the main result of this section.
\begin{theorem}\label{thm:conv_est}
Let $\{(\mu_k,u_k)\}_{k\geq 0}$ be the sequence of discrete eigenpairs and $\{\mathcal{M}_k\}_{k\geq0}$ the sequence of marked sets generated by Algorithm \ref{anfem}. Then the sequence of estimators $\{\eta_{k,1}(\mu_k,u_k) + \eta_{k,2}(u_k)\}_{k\geq0}$ generated by Algorithm \ref{anfem} converges to zero, i.e.,
    \begin{equation}\label{conv1_est}
        \lim_{k\to\infty}\eta_{k,1}(\mu_k,u_k) + \eta_{k,2}(u_k)=0.
    \end{equation}
    Hence, condition \eqref{condition 1} holds.
\end{theorem}

\begin{proof}
Let $\eta_{k}(\mu_k,u_k):=\eta_{k,1}(\mu_k,u_k) + \eta_{k,2}(u_k)$. Then it can be split into
    \begin{align}\label{pf1_conv_est}
        \eta_{k}(\mu_k,u_k)=\eta_{k}(\mu_k,u_k,\mathcal{T}_k\setminus\mathcal{T}_l^+)+
        \eta_{k}(\mu_k,u_k,\mathcal{T}_l^+)\quad  \text{with } k>l.
    \end{align}
    Similar to the proof of Lemma \ref{lem:max_est}, the upper bound  \eqref{disc_bd} and the normalization $\|u_k\|_{L^p(\Omega)}=1$ yield
        \begin{align*}
            \eta_{k}(\mu_k,u_k,\mathcal{T}_k\setminus\mathcal{T}_l^+)\leq
            &\sum_{T\in\mathcal{T}_k\setminus\mathcal{T}_l^+}\left(h_T^{q}\mu_k^q\| u_k\|_{L^{p}(T)}^{p}+\sum_{F\subset\p T}h_{F}\|[u_k]\|^p_{L^p(F)}\right)\\
            &\leq
            \sum_{T\in\mathcal{T}_k\setminus\T_l^+}h_{T}^{q}\mu_k^{q}\| u_k\|^p_{L^p(T)}+C\sum_{T\in\mathcal{T}_k\setminus\T_l^+}h_{T}^{p}\int_{D_T}|\bold{\nabla}_k u_k|^p\dx\\
            &\leq
            C\left(\max_{T\in\mathcal{T}_k\setminus\T_l^+}h_T^{q}\|u_k\|_{L^p(\Om)}^p+\max_{T\in\mathcal{T}_k\setminus\T_l^+}h_T^{p}\|u_k\|^{p}_{1,p,k}\right)\\
            &\leq C\max_{T\in\mathcal{T}_k\setminus\T_l^+}h_{T}^{\min(p,q)}.
        \end{align*}
Note that $\displaystyle{\max_{T\in\mathcal{T}_k\setminus\mathcal{T}_l^+}}h_T\leq\|h_l\chi_{l}^0\|_{L^\infty(\Om)}$ when $k>l$. Then the property \eqref{con_0_mz} implies 
     $\eta_k(\mu_k,u_k,\mathcal{T}_k\setminus\mathcal{T}_l^+)<\varepsilon/2$ for any given positive $\varepsilon$ when $l$ is sufficiently large. Note also that  $\mathcal{T}_l^+\subset\mathcal{T}_{k}^+\subset\mathcal{T}_k$ and $\mathcal{M}_k\cap\T_l^+=\emptyset$ for $k>l$. Then the marking requirement \eqref{marking} implies
    \begin{align*}
        \eta_{k}(\mu_k,u_k,\mathcal{T}_l^+) &\leq \#\mathcal{T}^+_l \max_{T\in\mathcal{T}^+_l}\eta_{k}(\mu_k,u_k,T)  \leq \#\mathcal{T}^+_l\max_{T\in\mathcal{M}_k}\eta_{k}(\mu_k,u_k,T) \\
        &\leq \# \mathcal{T}^+_l \left(\max_{T\in\mathcal{M}_k}\eta_{k,1}(\mu_k,u_k,T) + \max_{T\in\mathcal{M}_k}\eta_{k,2}(u_k,T)\right).
    \end{align*}
By Lemma \ref{lem:max_est}, we can choose some $K>l$ after fixing a large $l$ such that 
$\eta_{k}(\mu_k,u_k,\mathcal{T}_l^+)<\varepsilon/2$ for $k>K$, which directly leads to \eqref{conv1_est}.
\end{proof}

\section{Numerical tests}\label{sec:num}
To illustrate the performance of Algorithm \ref{anfem}, we consider two numerical examples on the unit square and the L-shaped domain. The SOLVE module employs a normalized inverse iteration of sublinear supersolutions (IISS) \cite[Algorithm 2]{MR2923523} to solve problem \eqref{min_disc} at each mesh level, cf. Algorithm \ref{alg:alg1}. Note that Algorithm \ref{alg:alg1} involves solving a $p$-Laplacian problem for the torsion function at Step 1 and to produce an inverse iteration sequence at Step 6, for which we employ a coordinate decomposition algorithm \cite{MR4653202} (cf. Algorithm \ref{alg:alg2}) with $f$ being the right hand side of the $p$-Laplacian problem in Steps 1 and 6 of Algorithm \ref{alg:alg1} and $g=0$.

\begin{algorithm}[htbp!]
\begin{algorithmic}[1]
	\State Solve (\textit{torsion funcion})
\begin{equation*}
\left\{
\begin{aligned}
-\Delta_pu_0:=-\mathrm{div}(|\boldsymbol{\nabla}u_0|^{p-2}\boldsymbol{\nabla}u_0)&=1&&\text{ in }\Omega,\\
 u_0&=0&&\text{ on }\partial\Omega.
\end{aligned}
\right.
\end{equation*}
\State $m\leftarrow 0$.
\State {$\lambda_{m} = 1/\|u_{m}\|_{L^\infty(\Omega)}^{p-1}$}. 
\While{$|\lambda_{m}-\lambda_{m-1}|/|\lambda_{m-1}|\geqslant\epsilon_{M}$}
    \State $m\leftarrow m+1$.
	\State Solve (\textit{inverse iteration})
    \begin{equation*}
        \left\{
            \begin{aligned}
                -\Delta_pu_{m}&=({u_{m-1}}/{\|u_{m-1}\|_{L^\infty(\Omega)}})^{p-1}\ &&{\rm in}\ \Omega,\\
                u_{m}&=0\ &&{\rm on} \ \partial\Omega.
            \end{aligned}
        \right.
    \end{equation*}
    \State  {$\lambda_{m} = 1/\|u_{m}\|_{L^\infty(\Omega)}^{p-1}$}. 
\EndWhile
	\State Return {$\lambda_m$} and {$u_{m}/{\|u_{m}\|_{L^{\infty}(\Omega)}}$}.\quad(\textit{first eigenvalue and first eigenfunction})
\end{algorithmic}
\caption{Normalized IISS \cite[Algorithm 2]{MR2923523}}
\label{alg:alg1}
\end{algorithm}

\begin{algorithm}[hbtp!]
    \begin{algorithmic}[1]
        \State Define two vector fields $\boldsymbol{\xi}_{1}$, $\boldsymbol{\nu}_{0}$: $\Omega \rightarrow \mathbb{R}^{2}$; $n\leftarrow 0$.
\While{$n=1$ or  $\frac{\|u_{n}-u_{n-1}\|_{L^2(\Omega)}}{\| u_{n-1}\|_{L^2(\Omega)}}>\epsilon_{N}$}
    \State $n\leftarrow n+1$.
 \State Compute  $u_{n}$  by solving a linear problem
\begin{equation*}
\left\{
\begin{aligned}
-\Delta u_{n}&=\nabla \cdot\left(\boldsymbol{\xi}_{n} - \boldsymbol{\nu}_{n-1}\right)+f ~&&\text{in}~\Omega,\\
u_{n}&=g~&&\text{on}~\partial \Omega.
\end{aligned}
\right.
\end{equation*}
 \State Compute $\boldsymbol{\nu}_{n}$  by solving the algebraic nonlinear equation $\left|\boldsymbol{\nu}_{n}\right|^{p-2} \boldsymbol{\nu}_{n}+\boldsymbol{\nu}_{n}=\boldsymbol{\xi}_{n}+\nabla u_{n}$.

  \State Compute $\boldsymbol{\xi}_{n+1}$ as $\boldsymbol{\xi}_{n+1} = \boldsymbol{\xi}_{n}+\nabla u_{n} - \boldsymbol{\nu}_{n}$.

     \EndWhile{}
 	
        \State Return ${u_{n}}$.

    \end{algorithmic}
\caption{Decomposition Coordination \cite{MR4653202}}
\label{alg:alg2}
\end{algorithm}

The MARK module of Algorithm \ref{anfem} employs D\"{o}rfler's strategy with $\theta=0.6$. It gives a subset $\mathcal{M}_k=\mathcal{M}_k^1\cup\mathcal{M}_k^2$ such that $\eta_{k,1}(\mu_k,u_k,\mathcal{M}_k^1)\geq 0.6 \eta_{k,1}(\mu_k,u_k)$ and $\eta_{k,2}(u_k,\mathcal{M}_k^2)\geq 0.6 \eta_{k,2}(\mu_k,u_k)$.
Algorithm \ref{anfem} proceeds until the relative error $|\mu_{k-1} - \mu_k|/ \mu_{k-1}$ is below a given tolerance $\epsilon_{K}$. To terminate the algorithm in finite loops, an upper bound $K$ is specified for the counter $k$ in Algorithm \ref{anfem}. We take $\epsilon_K:=10^{-4}$ and $K:=9$ in Algorithm \ref{anfem}, $\epsilon_M=\epsilon_N:=10^{-5}$ in Algorithms \ref{alg:alg1} and \ref{alg:alg2}. Moreover, each component of $\boldsymbol{\xi}_1$ and $\boldsymbol{\nu}_0$ of Algorithm \ref{alg:alg2} is independently sampled from the uniform distribution $U(0, 0.5)$.

To validate the lower bound \eqref{GLB} of the eigenvalue $\lambda_1$, we compute
\begin{align*}
e_{\mu}:=\lambda_{\op{ref}}^{1/p}-\frac{\mu_{\overline{k}}^{1/p}}{1 + 2\max_ {T\in\mathcal{T}_{\overline{k}}}\mathrm{diam}(T)\mu_{\overline{k}}^{1/p}},
\end{align*}
where $\overline{k}$ denotes the final iteration number in Algorithm \ref{anfem},
and the reference eigenpair $(\lambda_{\op{ref}},u_{\op{ref}})$ is computed by 
the conforming FEM on a sufficiently fine quasi-uniform triangular mesh.

\begin{example}\label{example2}
$\Omega=(0,1)^2$ is a unit square, and $p \in \{1.2, 1.5, 2, 2.5, 3, 4, 10, 20, 30\}$.
\end{example}

\begin{figure}[htb!]
\centering
\setlength{\tabcolsep}{0pt}
\begin{tabular}{cccc}
 \includegraphics[width=0.23\textwidth,trim={2.8cm 0.5cm 2.5cm 0cm},clip]{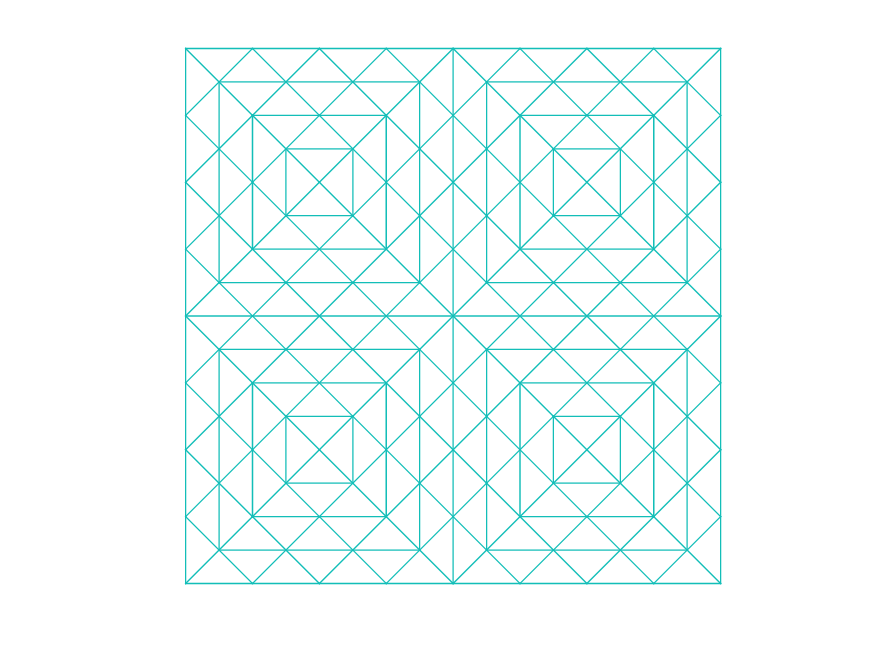}&
  \includegraphics[width=.23\textwidth,trim={2.8cm 0.5cm 2.5cm 0cm},clip]{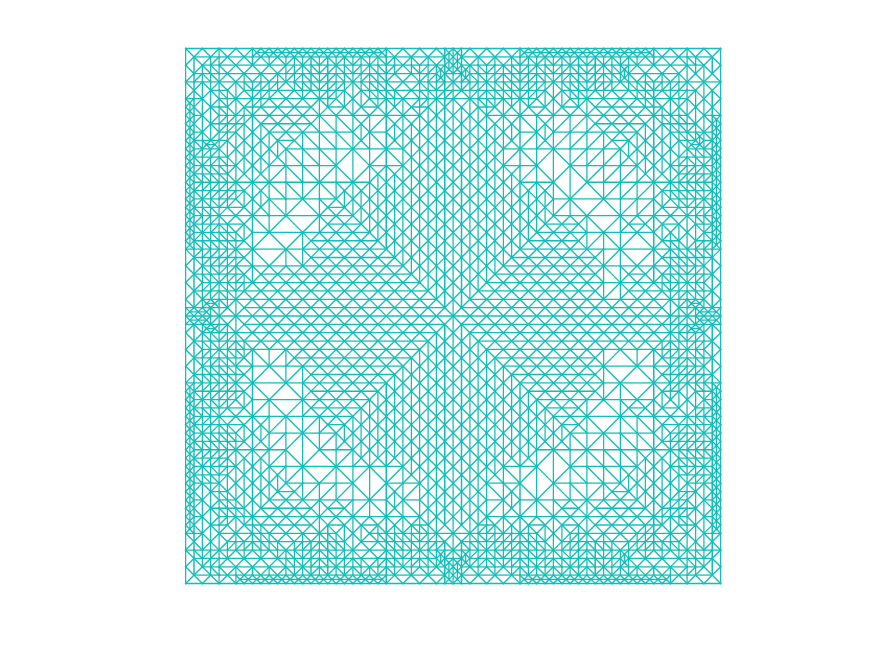} &
  \includegraphics[width=.23\textwidth,trim={2.8cm 0.5cm 2.5cm 0cm},clip]{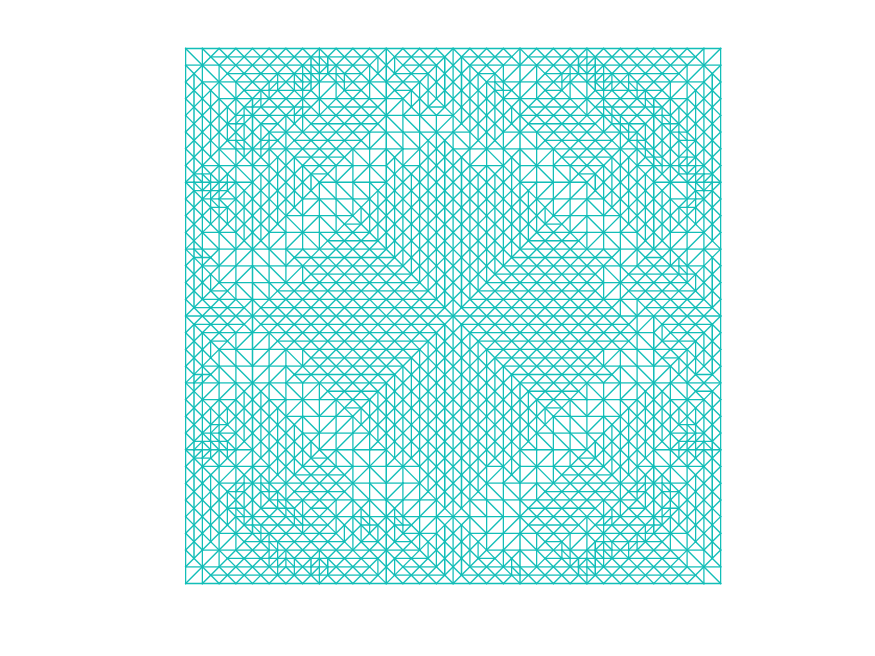} &
  \includegraphics[width=.23\textwidth,trim={2.8cm 0.5cm 2.5cm 0cm},clip]{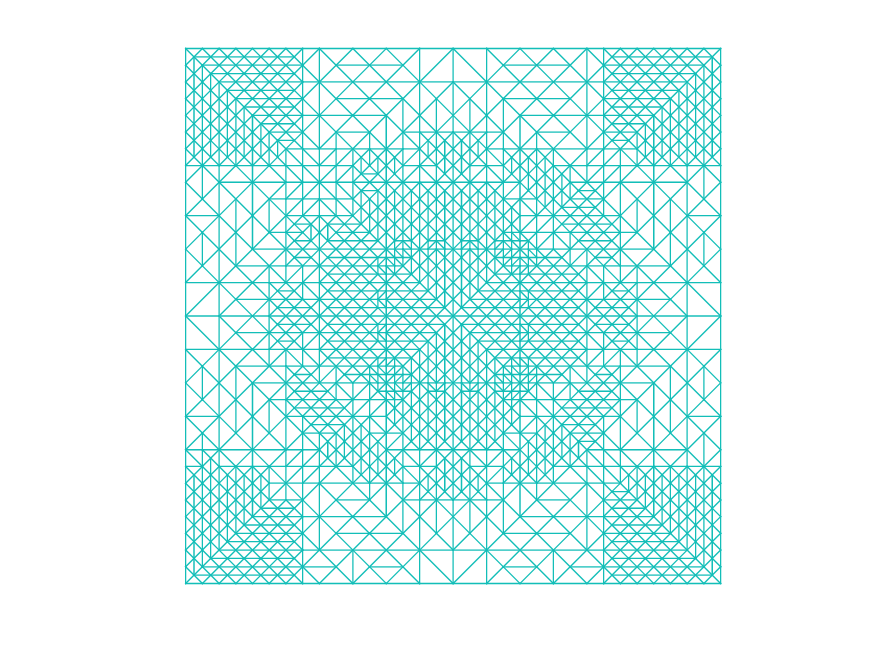}\\
  (a) $\mathcal{T}_0$ (400) & (b)$\mathcal{T}_3$ (7922) & (c) $\mathcal{T}_3$ (6100) & (d) $\mathcal{T}_3$ (4311) \\
   \includegraphics[width=0.23\textwidth,trim={2.8cm 0.5cm 2.5cm 0cm},clip]{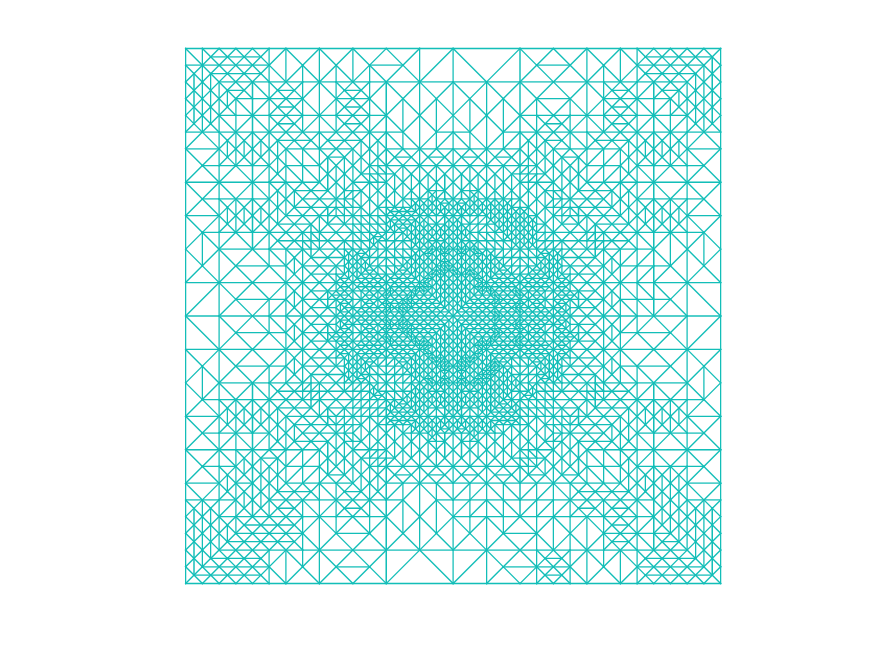}&
  \includegraphics[width=.23\textwidth,trim={2.8cm 0.5cm 2.5cm 0cm},clip]{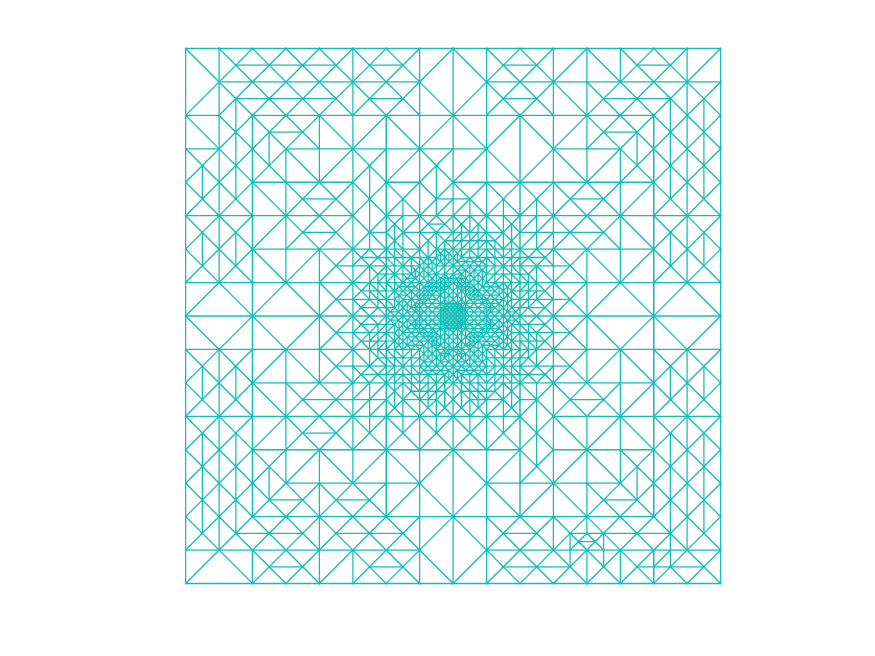} &
  \includegraphics[width=.23\textwidth,trim={2.8cm 0.5cm 2.5cm 0cm},clip]{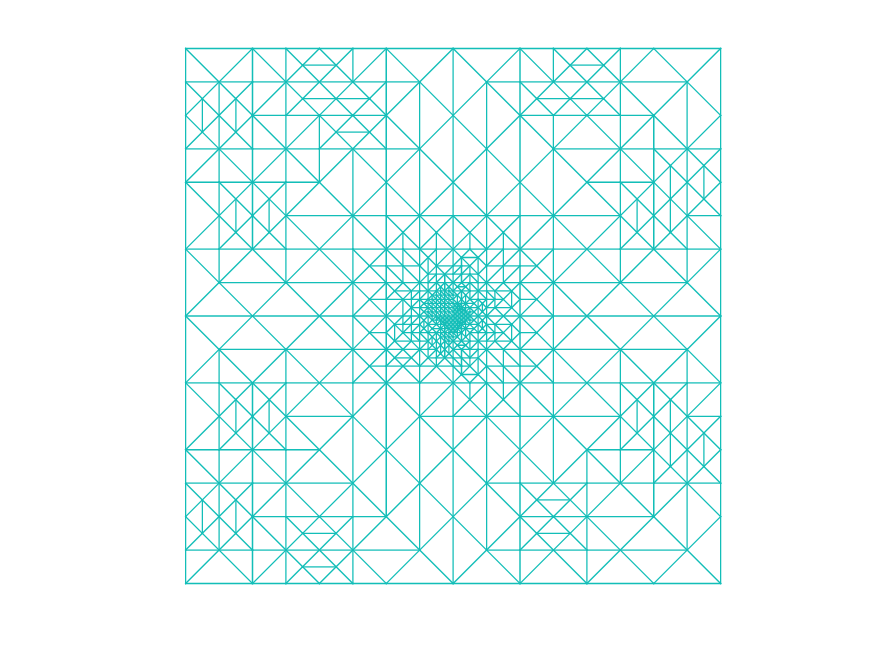} &
  \includegraphics[width=.23\textwidth,trim={2.8cm 0.5cm 2.5cm 0cm},clip]{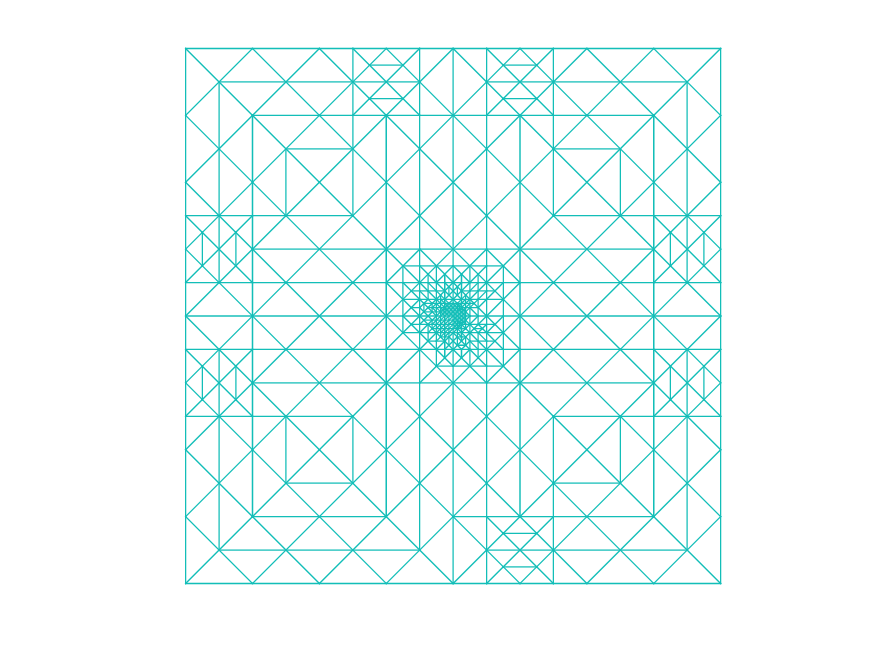}\\
  (e) $\mathcal{T}_4$ (7874) & (f)$\mathcal{T}_4$ (3244) & (g) $\mathcal{T}_4$ (1676) & (i) $\mathcal{T}_4$ (1223) 
\end{tabular}
\caption{The initial mesh $\mathcal{T}_0$ and the adaptively generated meshes after 3 refinement steps for $p\in\{1.2, 1.5, 2.5\}$ and 4 refinement steps for $p\in\{4, 10,20,30\}$ in Example \ref{example2}.
The numbers in the parentheses are dof.}
\label{fig:adptiveMesh=eg2}
\end{figure}

\begin{figure}[htb!]
    \centering
    \setlength{\tabcolsep}{0pt}
\begin{tabular}{ccc}
\includegraphics[width=.33\textwidth,trim={10cm 0cm 11cm 0.7cm},clip]{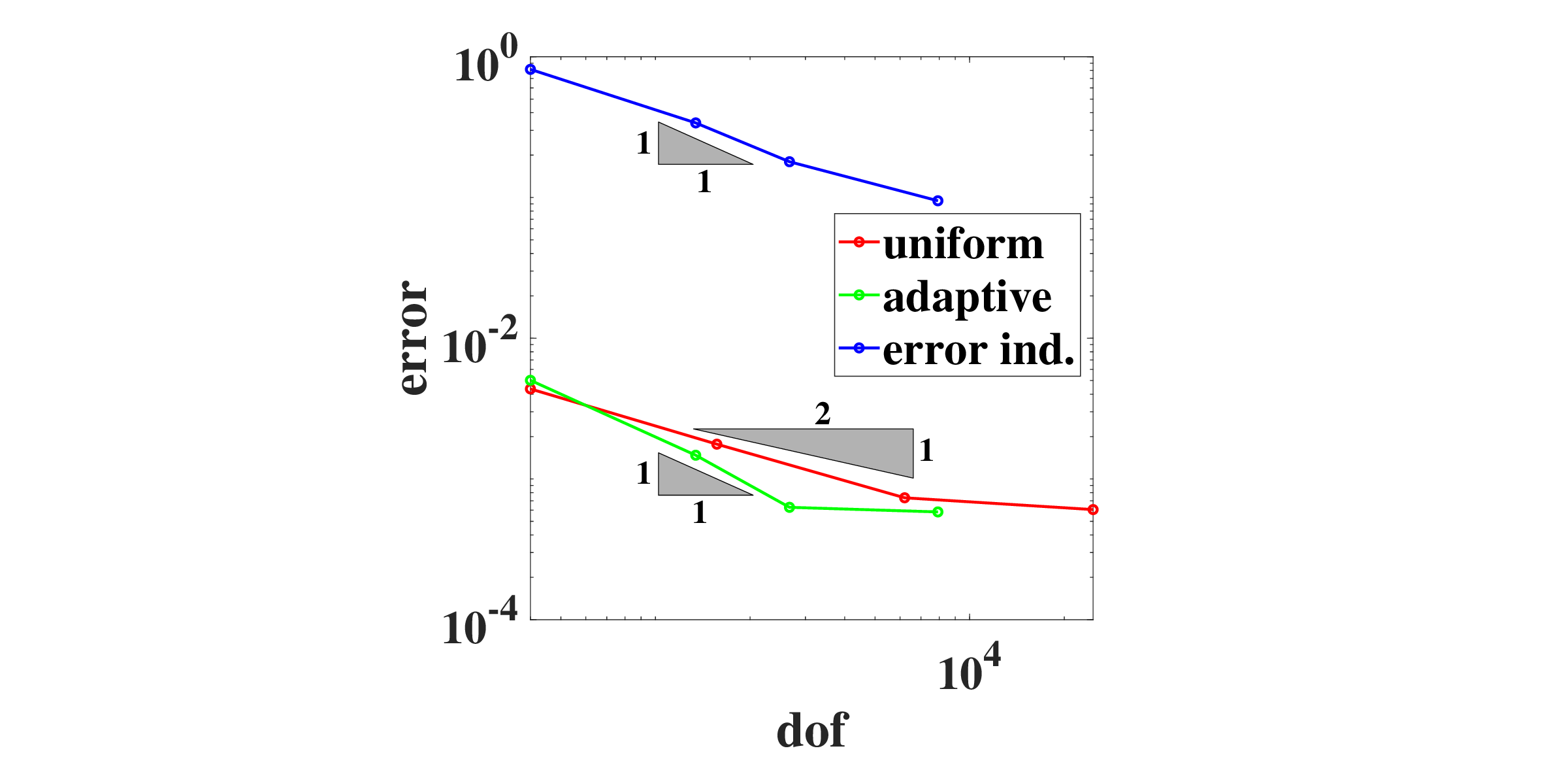} &
\includegraphics[width=.33\textwidth,trim={10cm 0cm 11cm 0.7cm},clip]{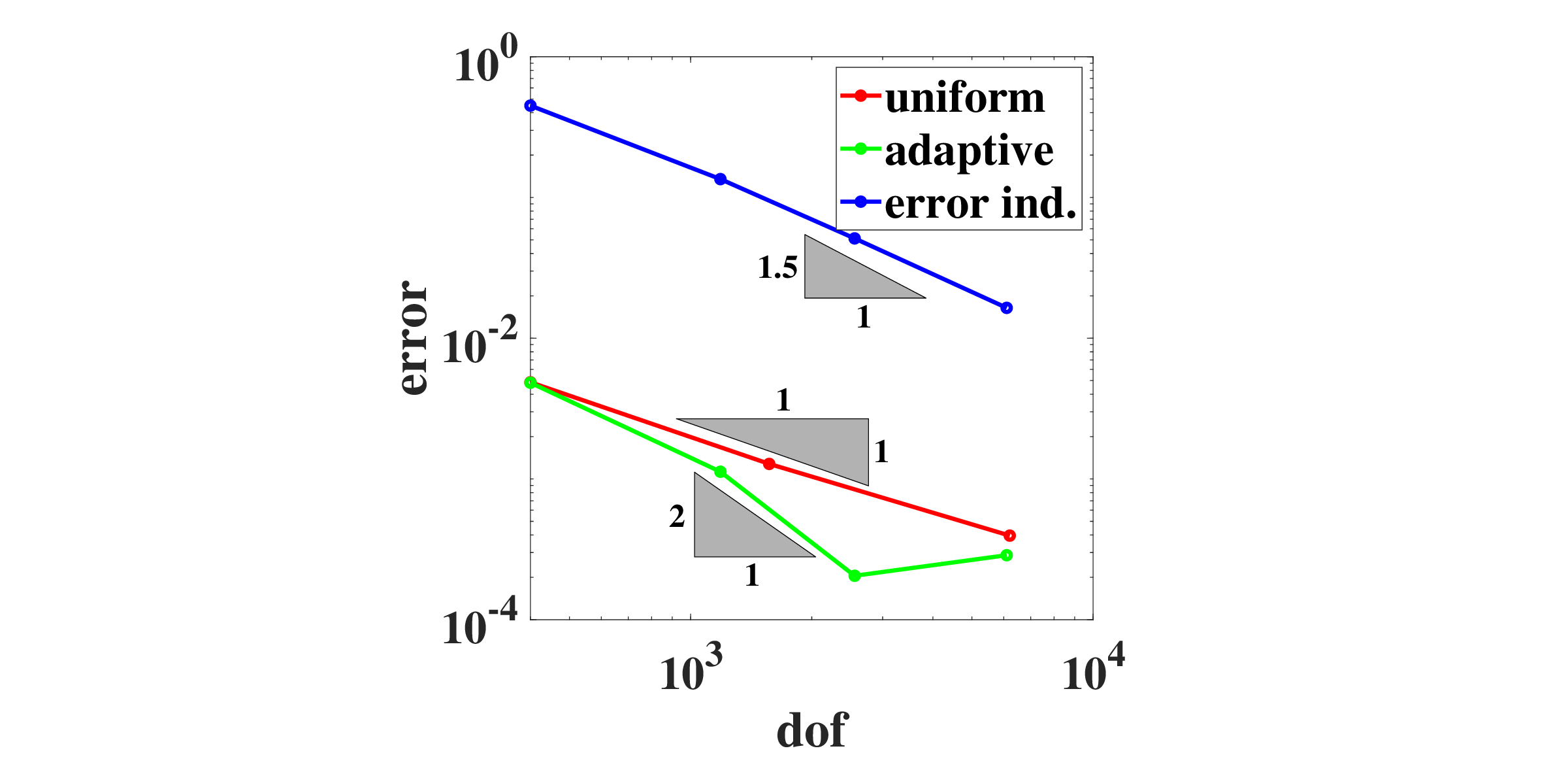} &
\includegraphics[width=.33\textwidth,trim={10cm 0cm 11cm 0.7cm},clip]{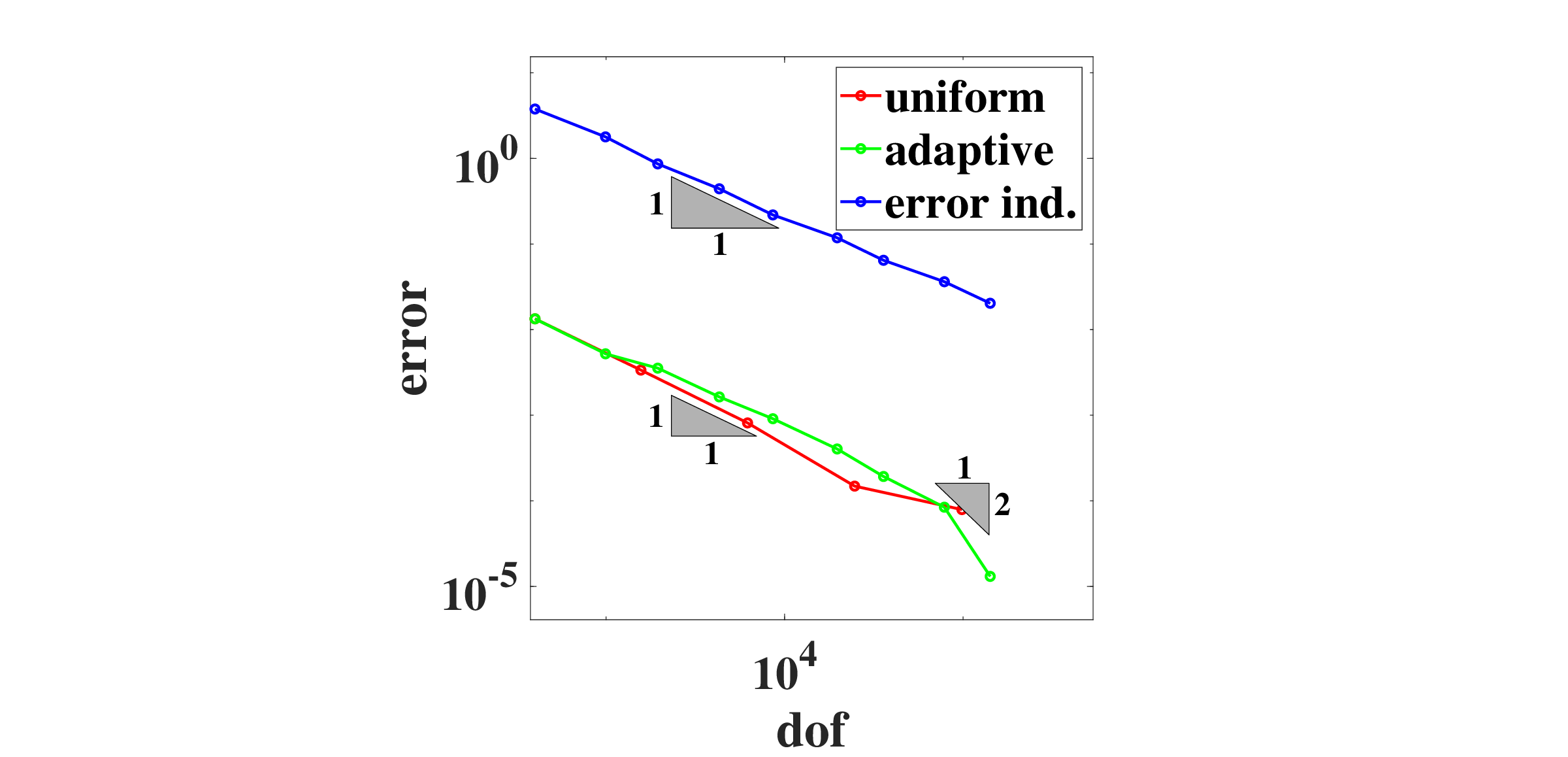}\\
(a) $p=1.2$ & (b) $p=1.5$ & (c) $p=2.5$
\end{tabular}
    \caption{The error estimator and relative error versus dof for Example \ref{example2} with $p\in\{1.2,1.5, 2.5\}$.}
    \label{fig:uniform-afem-eg2}
    \end{figure}
    
Starting from an initial mesh $\mathcal{T}_0$ (cf. Fig. \ref{fig:adptiveMesh=eg2}(a)), the sequence of approximate first eigenvalues $\{\mu_k\}_{k\geq 0}$ by Algorithm \ref{anfem} are reported in Table \ref{tab:UnitDiskEigenvalue} for $p\in\{1.2, 1.5, 2, 2.5, 3\}$ and in Table \ref{tab:UnitDiskEigenvalue102030} for $p\in\{4, 10, 20, 30\}$, respectively, where $k$ denotes the adaptive iteration number, and $\mu_k$ is the approximate first eigenvalue on the $k$-th adaptive mesh level. We also present the reference eigenvalue $\lambda_{\op{ref}}$.
The exact first eigenvalue of Laplacian ($p = 2$) on the unit square is $2\pi^2 \approx 1.973921\times10^1$, and the computed value $\mu_6 = 1.973800\times10^1$ has a relative error below $10^{-4}$, showing the high accuracy of Algorithm \ref{anfem}. It is observed that $\{\mu_k\}_{k \geq 0}$ is an increasing sequence and approaches the reference solution from below as $k$ increases for each $p$. These numerical evidences agree with the convergence result of Algorithm \ref{anfem} in Theorem \ref{thm:conv_est}. Moreover,  $e_{\mu}$ is positive for all cases, and its value lies in the interval $[0.869761,1.665345]$, indicating the robustness of the lower bound \eqref{GLB}.

\begin{table}[hbt!]
\centering
\caption{The quantitative result for Example \ref{example2} with $p\in\{1.2, 1.5, 2, 2.5, 3\}$: the number $k$ of adaptive loops, degrees of freedom and the computed first eigenvalue $\mu_k$.}
\label{tab:UnitDiskEigenvalue}
\resizebox{\textwidth}{!}{
\begin{tabular}{c cc cc cc cc cc}
    \toprule
    \multirow{2}{*}{$k$}&
    \multicolumn{2}{c}{$p=1.2$}&
    \multicolumn{2}{c}{$p=1.5$}&\multicolumn{2}{c}{$p=2$}&
    \multicolumn{2}{c}{$p=2.5$}&\multicolumn{2}{c}{$p=3$}\\
    \cmidrule(lr){2-11} 
     ~ & {dof} & $\mu_k$ & {dof} & $\mu_k$ & {dof} & $\mu_k$ & {dof} & $\mu_k$ & {dof} & $\mu_k$ \\
    \midrule
        0 & 400 & 6.169909 & 400 & 1.002415$\times10^{1}$ & 400 & 1.956992$\times10^{1}$ & 400 & 3.546683$\times10^{1}$  & 400 & 6.157986$\times10^{1}$  \\

        1 & 1343 & 6.191914 & 1186 & 1.006145$\times10^{1}$ & 1081  & 1.968469$\times10^{1}$ & 993 & 3.576048$\times10^{1}$ & 959 & 6.233010$\times10^{1}$ \\

        2 & 2670 & 6.197158 & 2556 & 1.007072$\times10^{1}$ & 2450  & 1.971080$\times10^{1}$ & 1949  & 3.582054$\times10^{1}$ & 1871 & 6.246818$\times10^{1}$ \\

        3 & 7922 & 6.197441 & 6100 & 1.006989$\times10^{1}$ & 4894  & 1.972702$\times10^{1}$ & 4311 & 3.588927$\times10^{1}$  & 4408 & 6.261686$\times10^{1}$ \\

        4 &  &  &   &   & 12416 & 1.973391$\times10^{1}$ & 8601 & 3.591541$\times10^{1}$  & 8313 & 6.268953$\times10^{1}$ \\

        5 &    &   &   &   & 22606 & 1.973649$\times10^{1}$ & 19718 & 3.593363$\times10^{1}$  & 19151 & 6.272179$\times10^{1}$  \\

        6 &    &   &   &   & 52742 & 1.973800$\times10^{1}$ & 35857  & 3.594122$\times10^{1}$ & 36663 & 6.271651$\times10^{1}$ \\

        7 &    &   &   &   &   &   & 78594 & 3.594511$\times10^{1}$ &  &   \\

        8 &    &   &    &   &    &   & 141972 & 3.594767$\times10^{1}$ &   &    \\
    \midrule
   \multirow{2}{*}{reference} & {dof} & $\lambda_{\op{ref}}$ &  {dof}   & $\lambda_{\op{ref}}$ & {dof} & $\lambda_{\op{ref}}$  & {dof} & $\lambda_{\op{ref}}$ & {dof} & $\lambda_{\op{ref}}$ \\ \cmidrule(lr){2-11}
   & 8321 &  6.201058 & 33025 & 1.007279$\times10^{1}$ & 131585 & 1.973932$\times10^{1}$ &  525313 & 3.594814$\times10^{1}$ & 131585 & 6.277145$\times10^{1}$ \\
   \midrule
    $e_{\mu}$ & \multicolumn{2}{c}{1.665345} &
               \multicolumn{2}{c}{1.361943} &
               \multicolumn{2}{c}{0.965675} &
               \multicolumn{2}{c}{0.869761} &
               \multicolumn{2}{c}{1.319564} \\
    \bottomrule
    \end{tabular}}
\end{table}

\begin{table}[hbt!]
\centering
\caption{The quantitative result for Example \ref{example2} with $p \in \{4, 10, 20, 30\}$: the number $k$ of adaptive loops, {degrees of freedom} and the computed first eigenvalue $\mu_k$.}
\label{tab:UnitDiskEigenvalue102030}
\resizebox{\textwidth}{!}{
\begin{tabular}{c cc cc cc cc}
    \toprule
    \multirow{2}{*}{$k$}&\multicolumn{2}{c}{$p=4$}&
    \multicolumn{2}{c}{$p=10$}&
    \multicolumn{2}{c}{$p=20$}&\multicolumn{2}{c}{$p=30$}\\
    \cmidrule(lr){2-9}
     ~ &{dof}  & $\mu_k$ &{dof} & $\mu_k$ &{dof} & $\mu_k$ &{dof} & $\mu_k$ \\
    \midrule
        0 &  400 & 1.712484$\times10^{2}$  & 400  & 3.210012$\times10^{4}$  &  400  & $7.645880\times10^{7}$ & 400  &  $1.427302\times10^{11}$  \\

        1 &  852 & 1.741399$\times10^{2}$  & 786  & 3.363833$\times10^{4}$  &  692  & $8.087508\times10^{7}$ & 551  &  $1.284860\times10^{11}$  \\

        2 &  1691 &  1.752220$\times10^{2}$ & 1112  & 3.417538$\times10^{4}$  &  795  & $8.130746\times10^{7}$ & 648  &  $1.285288\times10^{11}$  \\

        3 &  3952 & 1.757307$\times10^{2}$  & 1712  & 3.464071$\times10^{4}$  &  1015  & $8.159688\times10^{7}$ & 772  &  $1.325097\times10^{11}$  \\

        4 &  7874 & 1.762077$\times10^{2}$  & 3244  &  3.503375$\times10^{4}$ &  1676  & $8.521567\times10^{7}$ & 1223  &  $1.423945\times10^{11}$  \\

        5 &  16151 & 1.764255$\times10^{2}$  & 6933  &  3.521397$\times10^{4}$ &   3211 & $8.665902\times10^{7}$ & 2301  &  $1.526896\times10^{11}$  \\

        6 &  35039 & 1.763709$\times10^{2}$  & 15073  & 3.545912$\times10^{4}$  &  7338  & $8.877217\times10^{7}$ & 5087  &  $1.583267\times10^{11}$ \\

        7 &  69557 & 1.764861$\times10^{2}$  &  33322 & 3.550409$\times10^{4}$  &  16026  & $8.981353\times10^{7}$ &  11207 &   $1.605933\times10^{11}$ \\

        8 &  149599 & 1.765188$\times10^{2}$  & 73960  &  3.563539$\times10^{4}$ &  35654  & $9.071346\times10^{7}$ & 25523  &  $1.630840\times10^{11}$  \\

        9 & 294692  & 1.765924$\times10^{2}$  & 158267  &  3.566347$\times10^{4}$ & 80900  &  $9.127308\times10^{7}$ &   57416 &  $1.638727\times10^{11}$  \\
    \midrule
    \multirow{2}{*}{reference} & {dof} & $\lambda_{\op{ref}}$ &  {dof}   & $\lambda_{\op{ref}}$ & {dof} & $\lambda_{\op{ref}}$  & {dof} & $\lambda_{\op{ref}}$ \\ \cmidrule(lr){2-9}
    & 1355681 & 1.766239$\times10^{2}$ & 525313 & 3.580941$\times10^{4}$ &   131585 & 9.204824$\times10^{7}$  & 131585 & 1.650249$\times10^{11}$ \\
    \midrule
    $e_{\mu}$ & \multicolumn{2}{c}{1.141285} &
               \multicolumn{2}{c}{0.957366} &
               \multicolumn{2}{c}{0.962872} &
               \multicolumn{2}{c}{0.879252} \\
    \bottomrule
    \end{tabular}}
\end{table}

To verify the efficiency of Algorithm \ref{anfem} against uniform refinement, we plot in Fig. \ref{fig:uniform-afem-eg2} the convergence of error estimators, the relative errors $|\mu_k-\mu_{\op{ref}}|/\mu_{\op{ref}}$, and relative errors by the uniform refinement. Algorithm \ref{anfem} consistently outperforms uniform refinement. 
We refer to Fig. \ref{fig:adptiveMesh=eg2}(b)-(i) for an illustration of the adaptively generated meshes for $p\in\{1.2,1.5,2.5\}$ after 3 refinement steps and $p\in\{4,10,20,30\}$ after 4 refinement steps. It is observed that the marked element patch moves from the boundary to the domain center as the parameter $p$ grows. This behavior is consistent with the theoretical findings in \cite{MR2062882}: the first eigenvalue of the $p$-Laplacian converges to the Cheeger constant of $\Omega$ as $p \to 1^+$, and the associated eigenfunction approaches the characteristic function of the associated Cheeger domain. It also agrees with the result in \cite{MR4566532}, which asserts that the $\infty$-ground state, i.e. the limit of the first eigenfunction of the $p$-Laplacian as $p \to \infty$, is $\infty$-harmonic in the viscosity sense and further continuously differentiable \cite{MR2185662} except along two symmetric diagonal segments near the center.
\begin{example}\label{example3}
$\Omega=(0,2)^2\setminus [1,2)^2$ is L-shaped, with $p \in \{1.1, 1.2, 1.5, 2, 2.5, 3, 4, 10, 20, 30\}$.
\end{example}
\begin{figure}[hbt!]
\centering
\setlength{\tabcolsep}{0pt}
\begin{tabular}{cccc}
\includegraphics[width=.23\textwidth,trim={2.8cm 0.5cm 2.5cm 0cm},clip]{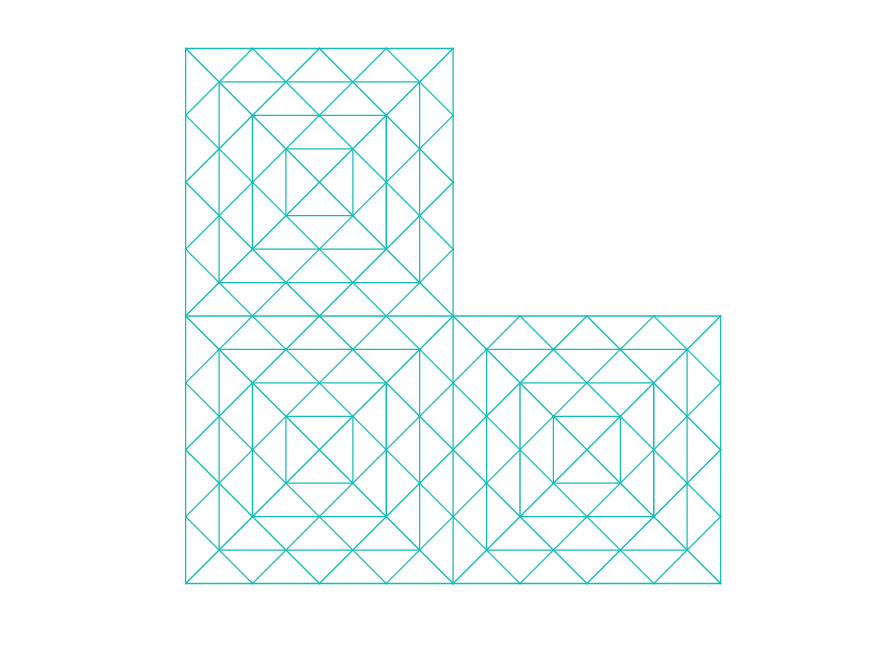}&
\includegraphics[width=.23\textwidth,trim={2.8cm 0.5cm 2.5cm 0cm},clip]{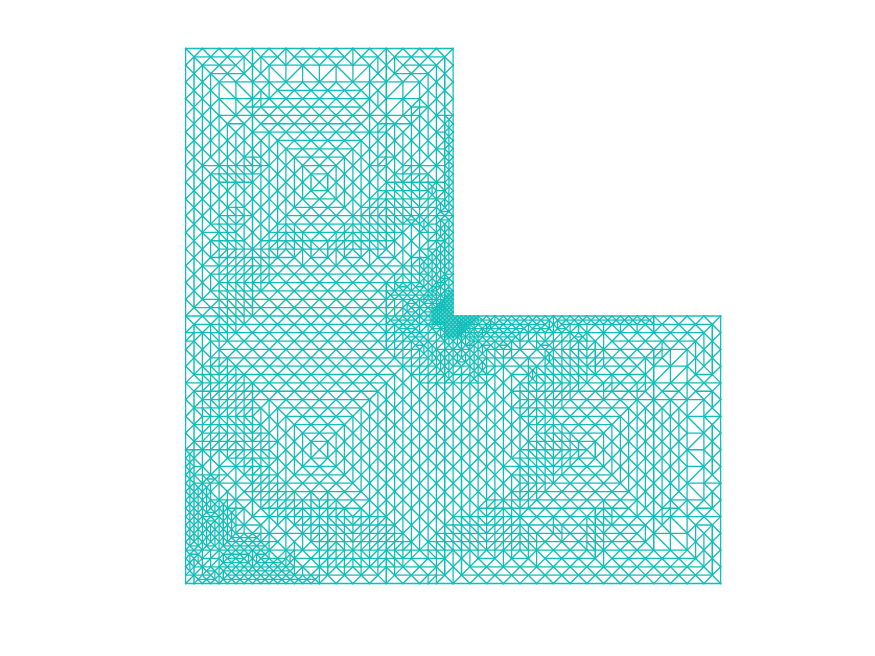} &
\includegraphics[width=.23\textwidth,trim={2.8cm 0.5cm 2.5cm 0cm},clip]{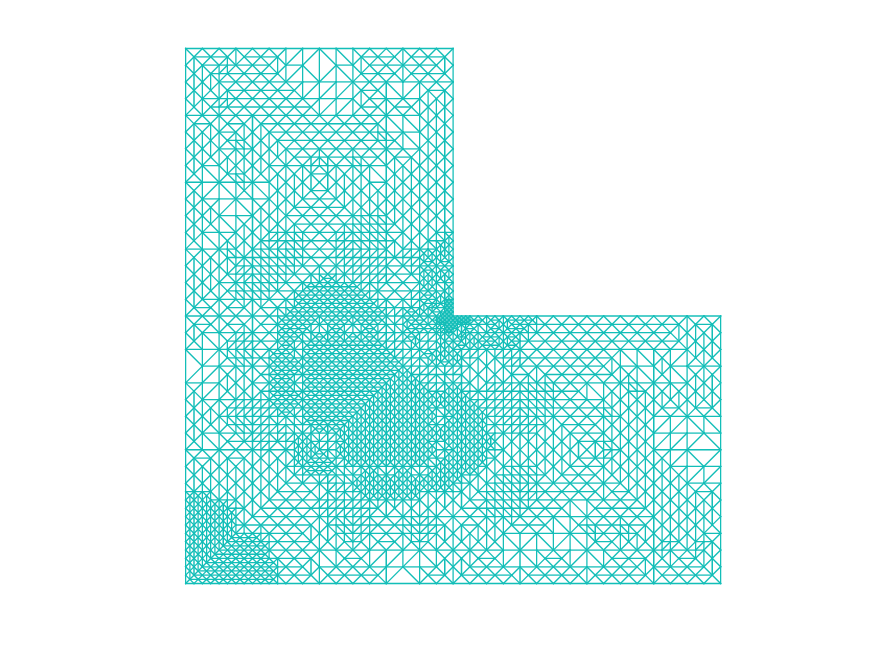} &
\includegraphics[width=.23\textwidth,trim={2.8cm 0.5cm 2.5cm 0cm},clip]{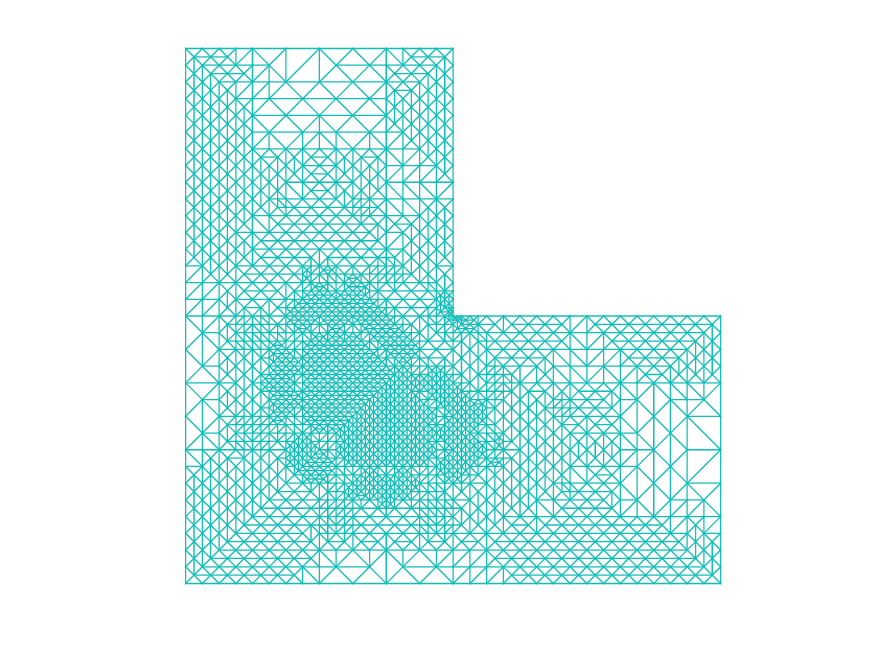}\\
(a) $\mathcal{T}_0$ (304) & (b) $\mathcal{T}_4$ (7333) & (c) $\mathcal{T}_4$ (8125) & (d) $\mathcal{T}_4$ (6491)
\end{tabular}
\caption{The initial mesh $\mathcal{T}_0$ and the adaptively generated meshes {after 4 refinement steps} for Example \ref{example3}: $p\in\{1.5, 2, 2.5\}$. 
The numbers in the parentheses denote the dof.}
\label{fig:adptiveMesh=eg3}
\end{figure}

\begin{figure}[hbt!]
\centering
\setlength{\tabcolsep}{0pt}
\begin{tabular}{ccc}
\includegraphics[width=.32\textwidth,trim={10cm 0cm 11cm 0.7cm},clip]{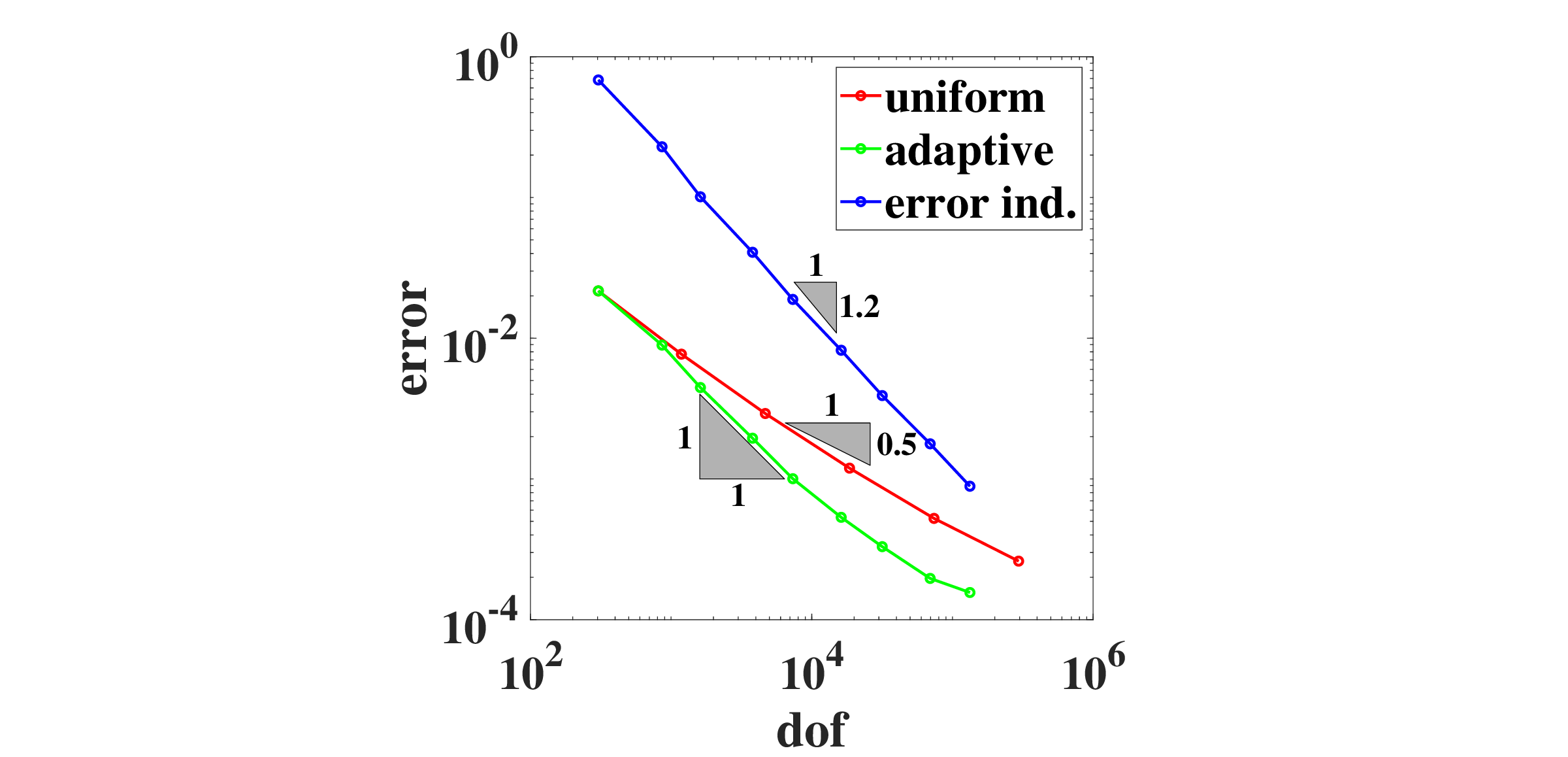} &
\includegraphics[width=.32\textwidth,trim={10cm 0cm 11cm 0.7cm},clip]{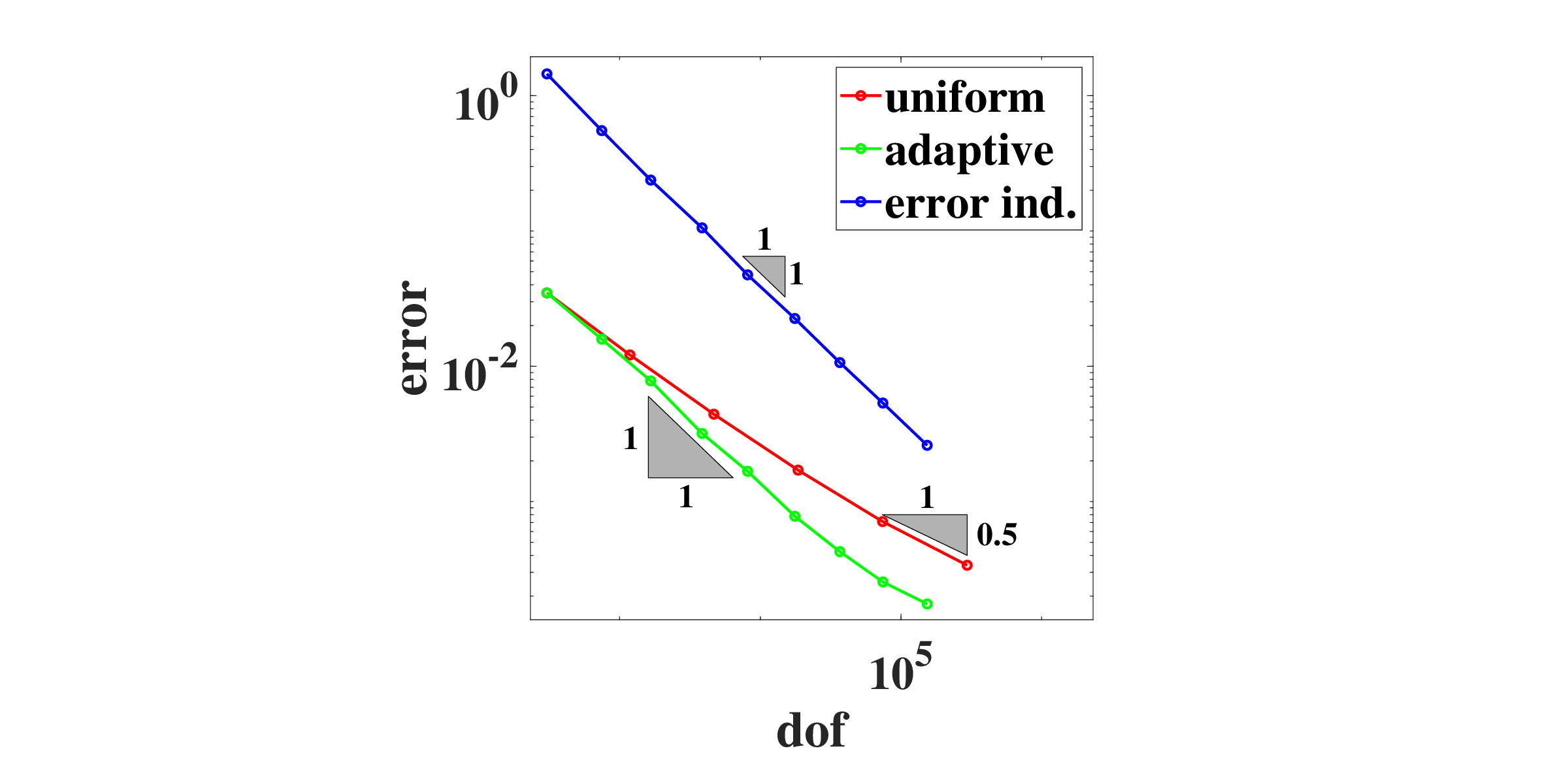} &
\includegraphics[width=.32\textwidth,trim={10cm 0cm 11cm 0.7cm},clip]{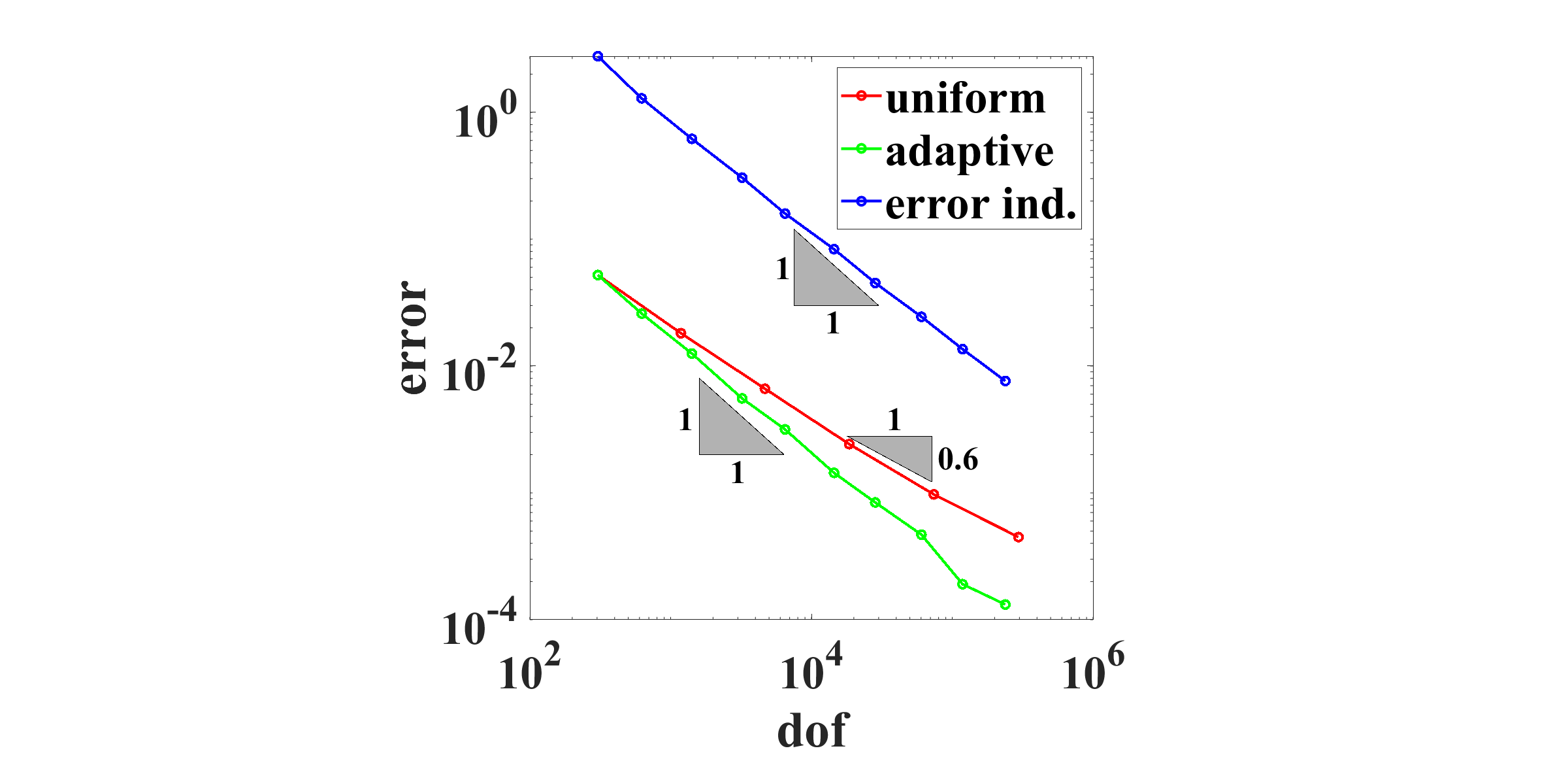}\\
(a) $p=1.5$ & (b) $p=2$ & (c) $p=2.5$
\end{tabular}
\caption{The error estimator and relative error versus dof. for Example \ref{example3} with $p\in\{1.5,2,2.5\}$.}
\label{fig:uniform-afem}
\end{figure}
\begin{table}[hbt!]
\centering
\caption{The quantitative result for Example \ref{example3} with $p\in\{1.1, 1.2, 1.5, 2, 2.5\}$: the number $k$ of adaptive loops, degrees of freedom and the computed first eigenvalue $\mu_k$.}
\label{tab:LshapedEigenvalue}
\resizebox{\textwidth}{!}{
\begin{tabular}{c cc cc cc cc cc}
    \toprule
    \multirow{2}{*}{$k$}&
    \multicolumn{2}{c}{$p=1.1$}&
    \multicolumn{2}{c}{$p=1.2$}&\multicolumn{2}{c}{$p=1.5$}&
    \multicolumn{2}{c}{$p=2$}&\multicolumn{2}{c}{$p=2.5$}\\
    \cmidrule(lr){2-11}
     ~ & {dof} & $\mu_k$ & {dof} & $\mu_k$ & {dof} & $\mu_k$ & {dof} & $\mu_k$ & {dof} & $\mu_k$ \\
    \midrule
        0 & 304 & 3.196434 & 304 & 3.795247 & 304 & 5.559635 & 304 & 9.304502 & 304  & 1.463779$\times10^{1}$ \\

        1 & 922  & 3.226163 & 939 & 3.833637 & 862 & 5.632225 & 746 & 9.487728 & 623 & 1.504322$\times10^{1}$ \\

        2 & 1572  & 3.236086 & 1783 & 3.847226 & 1616   & 5.657614 & 1663 & 9.565479 & 1415 &   1.524988$\times10^{1}$\\

        3 & 3153  & 3.247819 & 4436 & 3.853276 & 3793   & 5.671911 & 3854 & 9.609938 & 3217 & 1.535763$\times10^{1}$ \\

        4 & 7869  & 3.252765 & 8886 & 3.856288 & 7333   & 5.677275 & 8125 & 9.624524 & 6491 & 1.539451$\times10^{1}$ \\

        5 & 17870 & 3.252985 & 22524 & 3.857687 & 16190 & 5.679947 & 17594 & 9.633150 & 14474 & 1.542122$\times10^{1}$ \\

        6 &  &  & 46404 & 3.858336 & 31700 & 5.681104 & 36770 & 9.636549 & 28387 & 1.543044$\times10^{1}$ \\

        7 &  &   & 114493 & 3.858566 & 69471 & 5.681865 & 74341 & 9.638205 & 60428 & 1.543619$\times10^{1}$ \\

        8 &  &   &     &   & 133162 & 5.682095 & 153531 & 9.638974 & 118516 & 1.544048$\times10^{1}$ \\

        9 &   &  &     &   &     &   &   &  & 238274 & 1.544139$\times10^{1}$ \\
    \midrule
    \multirow{2}{*}{reference} & {dof} & $\lambda_{\op{ref}}$ &  {dof}   & $\lambda_{\op{ref}}$ & {dof} & $\lambda_{\op{ref}}$  & {dof} & $\lambda_{\op{ref}}$ & {dof} & $\lambda_{\op{ref}}$ \\ \cmidrule(lr){2-11}
   & 24833 & 3.257135  & 394241 & 3.859206 & 394241  & 5.682982 &  394241 & 9.640661  & 394241 &  1.544342$\times10^{1}$ \\
   \midrule
    $e_{\mu}$ & \multicolumn{2}{c}{1.236997} &
               \multicolumn{2}{c}{0.497906} &
               \multicolumn{2}{c}{0.528858} &
               \multicolumn{2}{c}{0.868283} &
               \multicolumn{2}{c}{0.812997} \\
    \bottomrule
    \end{tabular}}
\end{table}
The convergence history of Algorithm \ref{anfem} is reported in Tables \ref{tab:LshapedEigenvalue} and \ref{tab:LshapedEigenvalue102030}, which exhibits a similar convergence behavior of the computed first eigenvalues for each $p$. In particular, Table \ref{tab:LshapedEigenvalue} shows that the computed eigenvalue after 8 refinement steps is $\mu_8 = 9.638974$ for $p=2$ using Algorithm \ref{anfem}. The first eigenvalue was estimated to be $9.6397238389738806$ in \cite{MR4332791} using the $P_2$ Lagrange FEM on uniformly refined meshes combined with Aitken extrapolation, which yields a relative error of $7.778635 \times 10^{-5}$. Tables \ref{tab:LshapedEigenvalue} and \ref{tab:LshapedEigenvalue102030} indicate that the numerical eigenvalues by the adaptive method are smaller than the reference values and gradually increase toward the reference value for each $p$.

\begin{table}[hbt!]
\centering
\caption{The quantitative result for Example \ref{example3} with $p \in \{3, 4, 10, 20, 30\}$: the number $k$ of adaptive loops, {degrees of freedom} and the computed first eigenvalue $\mu_k$.}
\label{tab:LshapedEigenvalue102030}
\resizebox{\textwidth}{!}{
\begin{tabular}{c cc cc cc cc cc}
    \toprule
    \multirow{2}{*}{$k$}&\multicolumn{2}{c}{$p=3$}&\multicolumn{2}{c}{$p=4$}&
    \multicolumn{2}{c}{$p=10$}&
    \multicolumn{2}{c}{$p=20$}&\multicolumn{2}{c}{$p=30$}\\
    \cmidrule(lr){2-11}
     ~ & {dof} & $\mu_k$  & {dof} & $\mu_k$ & {dof} & $\mu_k$ & {dof} & $\mu_k$ & {dof} & $\mu_k$ \\
    \midrule
        0 & 304 & 2.222520$\times10^{1}$ & 304 & 4.777101$\times10^{1}$ & 304 & 1.948768$\times10^{3}$ & 304 & $3.283351\times10^{5}$ & 304 & $4.058800\times10^{7}$ \\

        1 & 575 & 2.313947$\times10^{1}$ & 819 & 5.190426$\times10^{1}$ & 517 & 2.733371$\times10^{3}$ & 349 & $4.16802.9\times10^{5}$ & 349 & $5.717465\times10^{7}$ \\

        2 & 1320 & 2.355887$\times10^{1}$ & 2359 & 5.353567$\times10^{1}$ & 1075   & 3.147725$\times10^{3}$ & 434 & $4.976239\times10^{5}$ & 435  & $8.240122\times10^{7}$ \\

        3 & 3077 & 2.379371$\times10^{1}$ & 6735 & 5.416681$\times10^{1}$ & 2285 & 3.411874$\times10^{3}$ & 582 & $7.969714\times10^{5}$ & 519 & $9.791527\times10^{7}$ \\

        4 & 6513 & 2.387736$\times10^{1}$ & 19364 & 5.442964$\times10^{1}$ & 4582 & 3.569684$\times10^{3}$ & 912 & $8.163835\times10^{5}$ & 680  & $1.360057\times10^{8}$ \\

        5 & 14331 & 2.393012$\times10^{1}$ & 54890 & 5.454355$\times10^{1}$ & 10365 & 3.670629$\times10^{3}$ & 1680 & $8.456558\times10^{5}$ & 1108 & $1.515244\times10^{8}$ \\

        6 & 31041 & 2.395957$\times10^{1}$ & 145951 & 5.458055$\times10^{1}$ & 22428 & 3.701425$\times10^{3}$ & 3717 & $8.655574\times10^{6}$ & 2148 & $1.706919\times10^{8}$ \\

        7 & 61365 & 2.397389$\times10^{1}$ & 374413 & 5.460941$\times10^{1}$ & 47056 & 3.742084$\times10^{3}$ & 8572 & $8.844211\times10^{5}$ & 4882 & $1.945276\times10^{8}$ \\

        8 & 131301 & 2.397841$\times10^{1}$ & 965543 & 5.460382$\times10^{1}$ & 99372 & 3.757267$\times10^{3}$ & 19274 & $1.117428\times10^{6}$ & 10912 & $2.014216\times10^{8}$ \\

        9 & 251406 & 2.398179$\times10^{1}$ & 2340956 & 5.460969$\times10^{1}$ & 214331 & 3.771393$\times10^{3}$ & 41795 & $1.169217\times10^{6}$ & 24729 & $2.490159\times10^{8}$  \\

    \midrule
    \multirow{2}{*}{reference} & {dof} & $\lambda_{\op{ref}}$ &  {dof}   & $\lambda_{\op{ref}}$ & {dof} & $\lambda_{\op{ref}}$  & {dof} & $\lambda_{\op{ref}}$ & {dof} & $\lambda_{\op{ref}}$ \\ \cmidrule(lr){2-11}
   & 394241 & 2.398973$\times10^{1}$  & 6295553  & 5.461138$\times10^{1}$ &  394241  & 3.794330$\times10^{3}$ &  98817 & 1.897308$\times10^{6}$ & 1574913  &  6.756470$\times10^{8}$\\
   \midrule
    $e_{\mu}$ & \multicolumn{2}{c}{0.764403} &
               \multicolumn{2}{c}{0.291581} &
               \multicolumn{2}{c}{0.828265} &
               \multicolumn{2}{c}{1.057466} &
               \multicolumn{2}{c}{0.993707} \\
    \bottomrule
    \end{tabular}}
\end{table}

To illustrate the advantage of Algorithm \ref{anfem} over uniform refinement, we plot in Fig. \ref{fig:uniform-afem} the convergence history of error estimator, the relative errors $|\mu_k-\mu_{\op{ref}}|/\mu_{\op{ref}}$, and relative errors by the uniform refinement. Adaptively generated meshes after 4 refinement steps are shown in Fig. \ref{fig:adptiveMesh=eg3}. As in Example \ref{example2},  Algorithm \ref{anfem} identifies the singularities of the solutions by performing additional local refinements.

\section{Conclusion}\label{sec:conclusion}
We propose an adaptive finite element method based on Crouzeix-Raviart elements to compute the first Dirichlet eigenpair of the $p$-Laplacian. The method relies on two error indicators that quantitatively characterize the residual and the nonconformity associated with the discrete eigenpair. We prove the strong convergence of the sequence of discrete eigenpairs generated by the adaptive algorithm. The numerical tests confirm the theoretical findings. The analysis crucially relies on the discrete compactness of Crouzeix-Raviart finite element eigenfunctions over a sequence of adaptively generated meshes. This interesting property may be useful in the study of adaptive nonconforming finite element methods for other nonlinear eigenvalue problems, e.g., the ground state of Bose-Einstein condensates.

\appendix
\section{\nopunct}\label{appendix}
In the appendix, we extend several results on the $L^2$-based Sobolev space $\bold{H}(\mathrm{div};\Om):=\{\bold{v}\in L^2(\Om)^d, \mathrm{div}\bold{v}\in L^2(\Om)\}$ in \cite{MR851383} (cf. Theorem 2.4, Theorem 2.5 and Corollary 2.8 on pages 27-28) to the $L^q$-case ($1<q<\infty$). Let $\Om\subset\mathbb{R}^d$ be an open bounded domain with a Lipschitz boundary $\partial\Om$. The $L^q$-version of $\bold{H}(\mathrm{div};\Om)$ is defined by
\[
    \bold{W}^{q}(\mathrm{div};\Om):=\{\bold{v}\in L^q(\Om)^d, \mathrm{div}\bold{v}\in L^q(\Om)\}
\]
equipped with the graph norm $\|\bold{v}\|_{\bold{W}^q(\mathrm{div})}:=(\|\bold{v}\|_{L^q(\Om)}^q+\|\mathrm{div}\bold{v}\|_{L^q(\Om)}^{q})^{1/q}$. For the space $W^{1,p}(\Om)$ with $p=q/(q-1)$, we also need its trace space $W^{1/q,p}(\p\Om)$ and dual space $W^{-1/q,q}(\p\Om)$.

\begin{theorem}\label{thm:density}
    The space $C^{\infty}(\overline{\Om})^d$ is dense in $\bold{W}^{q}(\mathrm{div};\Om)$.
\end{theorem}
\begin{proof}
    Let $\ell\in \bold{W}^q(\mathrm{div};\Om)'$, the dual space of $\bold{W}^q(\mathrm{div};\Om)$. By the Riesz representation theorem, there exists an $\boldsymbol{\ell}\in L^p(\Om)^{d}$ and an $\ell_{d+1}\in L^{p}(\Om)$ ($p=q/(q-1)$) such that
    \begin{equation}\label{pf1_thmA1}
        \ell(\bold{v})=\sum_{i=1}^d\int_{\Om}\ell_iv_i\dx+\int_{\Om}\ell_{d+1}\mathrm{div}\bold{v}\dx,\quad \forall \bold{v} \in \bold{W}^q(\mathrm{div};\Om).
    \end{equation}
    Now we assume that $\ell(\bold{v})=0$ for any $\bold{v}\in C^\infty(\overline{\Om})^d$ and let $\tilde{\ell}_i$ denote the extension of $\ell_i$ by zero outside $\Om$. Then \eqref{pf1_thmA1} can be rewritten as
    \begin{align*}
        \int_{\mathbb{R}^d}\left(\sum_{i=1}^d\tilde{\ell}_{i}v_i+\tilde{\ell}_{d+1}\mathrm{div}\bold{v}\right)\dx=0,\quad\forall \bold{v}\in C^\infty_0(\mathbb{R}^d)^d,
    \end{align*}
    which implies that in the sense of distributions in $\mathbb{R}^d$, $\tilde{\boldsymbol{\ell}}=\bold{\nabla}\tilde{\ell}_{d+1}$,
    with $\tilde{\boldsymbol{\ell}}=(\tilde{\ell}_1,\ldots,\tilde{\ell}_d)^T$. Hence $\tilde{\ell}_{d+1}\in W^{1,p}(\mathbb{R}^d)$ since $\tilde{\boldsymbol{\ell}}\in L^p(\mathbb{R}^d)^d$. As $\tilde{\ell}_{d+1}$ is also the zero extension of $\ell_{d+1}$, it follows from \cite[Theorem 1.2, Chapter I]{MR851383} (also cf. \cite{MR2424078}) that $\ell_{d+1}\in W^{1,p}_{0}(\Om)$. The density of  $C^{\infty}_0(\Om)$ in $W^{1,p}_0(\Om)$ admits that there exists a sequence $\{\psi_m\}_{m\geq 1}\subset C^\infty_0(\Om)$ converging to $\ell_{d+1}$ in $W^{1,p}$-norm, which, along with the distributional definition of $\mathrm{div}$, yields
    \begin{align*}
        0=\lim_{m\to\infty}\int_{\Om}\left(\bold{\nabla}\psi_m\cdot\bold{v}+\psi_m\mathrm{div}\bold{v}\right)\dx
        =\int_{\Om}\left(\bold{\nabla}\ell_{d+1}\cdot\bold{v}+\ell_{d+1}\mathrm{div}\bold{v}\right)\dx
        =\ell(\bold{v}),\quad\forall \bold{v}\in \bold{W}^q(\mathrm{div};\Om).
    \end{align*}
    Therefore $\ell$ also vanishes on $\bold{W}^q(\mathrm{div};\Om)$. The desired density follows.
\end{proof}

\begin{theorem}\label{thm:trace}
The normal-component trace operator $\gamma_n:\bold{v}\mapsto\bold{v}\cdot\bold{n}$ is continuous from $\bold{W}^q(\mathrm{div};\Om)$ to $W^{-1/q,q}(\p\Om)$.
\end{theorem}
\begin{proof}
Note that the following integration by parts formula holds
\begin{align*}
    \int_{\Om}\bold{v}\cdot\bold{\nabla}\psi \dx+\int_{\Om}\mathrm{div}\bold{v}\psi \dx
    =\int_{\partial\Om}\bold{v}\cdot\bold{n}\psi \ds,\quad\forall \bold{v}\in C^\infty(\overline{\Om})^d,\psi\in C^\infty(\overline{\Om}).
\end{align*}
Since $C^\infty(\overline{\Om})$ is dense in $W^{1,p}(\Om)$ $(p=q/(q-1))$, the identity defines a linear functional on $W^{1,p}(\Om)$ satisfying
\[
    \left|\int_{\partial\Om}\bold{v}\cdot\bold{n}\psi \ds\right|\leq\|\bold{v}\|_{\bold{W}^q(\mathrm{div})}\|\psi\|_{W^{1,p}(\Om)},\quad\forall \bold{v}\in C^\infty(\overline{\Om})^d,\psi\in W^{1,p}(\Om).
\]
Note that the trace mapping $\gamma: W^{1,p}(\Omega)\to W^{1/q,p}(\p\Om)$ is bounded and surjective. Then the open mapping theorem implies
\begin{equation}\label{eqn:inf-trace}
\inf_{v\in\mathrm{ker}\gamma}\|\psi +v\|_{W^{1,p}(\Omega)} \leq C \|\gamma\psi\|_{W^{1-1/p,p}(\p\Omega)}, \quad \forall\psi \in W^{1,p}(\Omega).
\end{equation}
Since $\mathrm{ker}\gamma=W_0^{1,p}(\Omega)$ is a closed subspace of $W^{1,p}(\Omega)$, by the uniform convexity of $W^{1,p}(\Omega)$ ($1<p<\infty$), we deduce that the infimum in \eqref{eqn:inf-trace} is attained at a unique $v\in W_0^{1,p}(\Omega)$. Thus for any $g\in W^{1/q,p}(\p\Om)$, there exists a $\psi\in W^{1,p}(\Om)$ such that $\psi=g$ on $\p\Om$ and
\[
       \left|\langle\bold{v}\cdot\bold{n},g\rangle_{W^{-1/q,q}(\p\Om),W^{1/q,p}(\p\Om)}\right|\leq  \|\bold{v}\|_{\bold{W}^q(\mathrm{div})}\|\psi\|_{W^{1,p}(\Om)}
       \leq C\|\bold{v}\|_{\bold{W}^q(\mathrm{div})}\|g\|_{W^{1-1/p,p}(\p\Om)},
\]
for any $\bold{v}\in C^\infty(\overline{\Om})^d$. Thus
$\|\bold{v}\cdot\bold{n}\|_{W^{-1/q,q}(\p\Om)}\leq
    C\|\bold{v}\|_{\bold{W}^q(\mathrm{div};\Om)}$, and the linear mapping $\gamma_n:\bold{v}\mapsto\bold{v}\cdot\bold{n}$ is continuous from $C^\infty(\overline{\Om})^d$ to $W^{-1/q,q}(\p\Om)$. Since $C^\infty(\overline{\Om})^d$ is dense in $\bold{W}^{q}(\mathrm{div};\Om)$ (cf. Theorem \ref{thm:density}), we may extend $\gamma_n$ continuously, still denoted by $\gamma_n$, from $\bold{W}^{q}(\mathrm{div};\Om)$ to $W^{-1/q,q}(\p\Om)$.
\end{proof}

\begin{remark}
Theorems \ref{thm:density} and \ref{thm:trace} imply the following Green's formula: for any $\bold{v}\in \bold{W}^q(\mathrm{div};\Om)$, $\psi\in W^{1,p}(\Om)$ with $1/p+1/q=1$,
    \begin{equation}\label{Greenform}
        \int_{\Om}\bold{v}\cdot\bold{\nabla}\psi \dx+\int_{\Om}\mathrm{div}\bold{v}\psi \dx=\langle\bold{v}\cdot\bold{n},\psi\rangle_{W^{-1/q,q}(\p\Om),W^{1/q,p}(\p\Om)}.
    \end{equation}
\end{remark}

\begin{theorem}\label{thm:surjective}
The normal-component trace operator $\gamma_n:\bold{v}\mapsto\bold{v}\cdot\bold{n}$ is
surjective from $\bold{W}^q(\mathrm{div};\Om)$ to $W^{-1/q,q}(\p\Om)$.
\end{theorem}
\begin{proof}
Consider the following $p$-Laplacian problem: given $g\in W^{-1/q,q}(\p\Om)$, find $\phi\in W^{1,p}(\Om)$ such that
\begin{equation}\label{p-Lap_aux}
    \left\{
    \begin{aligned}
        -\mathrm{div}(|\boldsymbol{\nabla}\phi|^{p-2}\boldsymbol{\nabla}\phi)+|\phi|^{p-2}\phi&=0&&\mbox{in}~\Om,\\
        |\bold{\nabla}\phi|^{p-2}\partial_n\phi&=g&&\mbox{on}~\p\Om.
    \end{aligned}
    \right.
\end{equation}
Let $\mathcal{I}(\psi)=\int_{\Om}|\bold{\nabla}\psi|^p\dx/p+\int_{\Om}|\psi|^p\dx/p-\langle g,\psi\rangle_{W^{-1/q,q}(\p\Om),W^{1/q,p}(\p\Om)}$. Then problem \eqref{p-Lap_aux} is equivalent to a minimization problem: find $\phi\in W^{1,p}(\Om)$ such that
\begin{equation}\label{p-Lap_aux_min}
    \mathcal{I}(\phi)=\inf_{\psi\in W^{1,p}(\Om)}\mathcal{I}(\psi).
\end{equation}
The trace theorem implies that on $W^{1,p}(\Om)$
    \begin{align*}
    \mathcal{I}(\psi)&\geq(\|\psi\|_{L^{p}(\Om)}^p+\|\bold{\nabla}\psi\|_{L^{p}(\Om)}^p)/p-\|g\|_{W^{-1/q,q}(\p\Om)}\|\psi\|_{W^{1/q,p}(\p\Om)}\\
    &\geq (\|\psi\|_{L^{p}(\Om)}^p+\|\bold{\nabla}\psi\|_{L^{p}(\Om)}^p)/p - C \|g\|_{W^{-1/q,q}(\p\Om)}\|\psi\|_{W^{1,p}(\Om)}.
    \end{align*}
Since $p>1$, $\mathcal{I}$ is coercive. Now assume that $\psi_n\to\psi$ weakly in $W^{1,p}(\Om)$. The weak convergence also holds in $W^{1/q,p}(\partial\Om)$ due to the continuity of the trace operator from $W^{1,p}(\Om)$ to $W^{1/q,p}(\p\Om)$. So we have
\[
    \mathcal{I}(\psi)\leq\liminf_{n\to\infty}\mathcal{I}(\psi_n),
\]
i.e., $\mathcal{I}$ is weakly lower semi-continuous in $W^{1,p}(\Om)$. Clearly, $\mathcal{I}$ is strictly convex. By the direct method in calculus of variations \cite{MR2480025}, problem \eqref{p-Lap_aux_min} has a unique minimizer. This gives a unique weak solution of problem \eqref{p-Lap_aux}.
By setting $\bold{v}=|\bold{\nabla}\phi|^{p-2}\bold{\nabla}\phi\in L^q(\Om)^d$, we get from \eqref{p-Lap_aux} that $\mathrm{div}\bold{v}=|\phi|^{p-2}\phi\in L^q(\Om)$ and $\bold{v}\cdot\bold{n}=g$. That is, for any $g\in W^{-1/q,q}(\p\Om)$ there exists a $\bold{v}\in\bold{W}^q(\mathrm{div};\Om)$ such that $\bold{v}\cdot\bold{n}=g$ on $\p\Om$.
\end{proof}

\begin{remark}
The variational formulation of problem \eqref{p-Lap_aux} further implies
\[
    \|\phi\|^p_{W^{1,p}(\Om)}=\langle g,\phi\rangle_{W^{-1/q,q}(\p\Om),W^{1/q,p}(\p\Om)}\leq C\|g\|_{W^{-1/q,q}(\p\Om)}\|\phi\|_{W^{1,p}(\Om)}.
\]
Since $\bold{v}=|\bold{\nabla}\phi|^{p-2}\bold{\nabla}\phi$, $\mathrm{div}\bold{v}=|\phi|^{p-2}\phi$ and $q=p/(p-1)$, we obtain
\begin{align*}
    \|\bold{v}\|_{\bold{W}^q(\mathrm{div})}\leq C\|g\|_{W^{-1/q,q}(\p\Om)}=C\|\bold{v}\cdot\bold{n}\|_{W^{-1/q,q}(\Om)},
\end{align*}
which, together with Theorems \ref{thm:trace} and \ref{thm:surjective}, further implies that $\gamma_n: \bold{W}^{q}(\mathrm{div};\Omega) \to W^{-1/q,q}(\partial\Omega)$ is an isomorphism with a bounded inverse.
\end{remark}

\bibliographystyle{siam}
\bibliography{reference}

\end{document}